\documentclass{amsart}
\usepackage{latexsym,amsxtra,amscd,ifthen}
\usepackage{amsfonts}
\usepackage{verbatim}
\usepackage{amsmath}
\usepackage{amsthm}
\usepackage{amssymb}
\numberwithin{equation}{section}

\theoremstyle{plain}
\newtheorem{theorem}{Theorem}[section]
\newtheorem{lemma}[theorem]{Lemma}
\newtheorem{proposition}[theorem]{Proposition}
\newtheorem{corollary}[theorem]{Corollary}

\theoremstyle{definition}
\newtheorem{definition}[theorem]{Definition}
\newtheorem{example}[theorem]{Example}

\newtheorem{remark}[theorem]{Remark}
\newtheorem{question}[theorem]{Question}

\makeatletter              
\let\c@equation\c@theorem  
\makeatother

\newcommand{\bfl}{\mathfrak l}

\DeclareMathOperator{\hdet}{hdet} 
\DeclareMathOperator{\gldim}{gldim}

\DeclareMathOperator{\Ext}{Ext} \DeclareMathOperator{\Tor}{Tor}

 \DeclareMathOperator{\tr}{tr}

\DeclareMathOperator{\Aut}{Aut}

\DeclareMathOperator{\injdim}{injdim}

\newcommand{\kpx}{k[\underline{x}]}
\newcommand{\knx}{k_{-1}[\underline{x}]}
\newcommand{\kqx}{k_{q}[\underline{x}]}

\newcommand{\cal}{\mathcal}

\newcommand{\be}{\begin{enumerate}}
\newcommand{\ee}{\end{enumerate}}
\newcommand{\bq}{\begin{eqnarray*}}
\newcommand{\eq}{\end{eqnarray*}}
\newcommand{\bqn}{\begin{eqnarray}}
\newcommand{\eqn}{\end{eqnarray}}

\begin{document}
\title[Invariants of $(-1)$-skew polynomial rings]
{Invariants of $(-1)$-skew polynomial rings \\
under permutation representations}

\author{E. Kirkman, J. Kuzmanovich and J.J. Zhang}

\address{Kirkman: Department of Mathematics,
P. O. Box 7388, Wake Forest University, Winston-Salem, NC 27109}

\email{kirkman@wfu.edu}

\address{Kuzmanovich: Department of Mathematics,
P. O. Box 7388, Wake Forest University, Winston-Salem, NC 27109}

\email{kuz@wfu.edu}

\address{Zhang: Department of Mathematics, Box 354350,
University of Washington, Seattle, Washington 98195, USA}

\email{zhang@math.washington.edu}









\maketitle


\bigskip

\setcounter{section}{-1}
\section{Introduction}
\label{xxsec0}

Let $k$ be a base field of characteristic zero (unless otherwise
stated) and let $\kqx$ denote the $q$-skew polynomial ring $k_{q}[x_1,
\ldots, x_n]$ that is generated by $\{x_i\}_{i=1}^n$ and subject to
the relations $x_jx_i=qx_ix_j$ for all $i<j$, where $q$ is a nonzero
element in $k$.  In previous work \cite{KKZ1}-\cite{KKZ4} we have
studied the invariant theory of noncommutative Artin-Schelter
regular (or AS regular, for short) algebras such as $\kqx$ under
linear actions by finite groups $G$. We have shown that often the
classical invariant theory of the commutative AS regular algebra
$\kpx:=k[x_1, \ldots, x_n]$ extends to noncommutative AS regular
algebras in some analogous way. In this paper we consider the case
where $G$ is a group of permutations of $\{x_i\}_{i=1}^n$ acting on
the $(-1)$-skew polynomial ring $\knx$, which is generated by
$\{x_i\}_{i=1}^n$ and subject to the relations
\begin{equation}
\label{E0.0.1}\tag{E0.0.1} x_ix_j = - x_j x_i
\end{equation}
for all $i \neq j$. We have chosen to consider $\knx$ because any
permutation of $\{x_i\}_{i=1}^n$ preserves the relations
\eqref{E0.0.1}, and hence extends to an algebra automorphism of
$\knx$; the only $q$-skew polynomial algebras $\kqx$ with this property
are the cases when $q= \pm 1$. Hence any subgroup of the symmetric
group $\mathfrak{S}_n$ acts on both $\kpx$ and $\knx$ as
permutations, and our main focus is on the ring of invariants $\knx^G$
when $G$ is a subgroup of $\mathfrak{S}_n$.

The study of the fixed subring $\kpx^G$ under permutation groups $G$
of the commutative indeterminates $\{x_i\}_{i=1}^n$ has a long and
distinguished history.  Gauss showed that when $G$ is the full
symmetric group $\mathfrak{S}_n$, invariant polynomials could be
expressed uniquely in terms of the $n$ symmetric polynomials
\cite[Theorem 4.13]{Ne}; the symmetric polynomials are algebraically
independent so that $\kpx^G$ is itself a polynomial ring.  This
result was generalized to other groups (so-called ``reflection
groups") by Shephard-Todd \cite{ST} and Chevalley \cite{Ch} in the
1950s. It follows from \cite[Theorem 1.1]{KKZ2} that $\knx^G$ will
not be an AS regular algebra, even for a classical reflection group
like the symmetric group. However, we will show that $A^G$ is always
an AS Gorenstein domain [Theorem \ref{xxthm1.5}], while $\kpx^G$ is
not always Gorenstein [Example \ref{xxex1.6}]. In \cite{CA} algebra
generating sets for $\knx^{\mathfrak{S}_n}$, the invariants under
the full symmetric group [Theorem \ref{zzthm3.10}], and for
$\knx^{\mathfrak{A}_n}$, the invariants under the alternating group
$\mathfrak{A}_n$ [Theorem \ref{zzthm4.10}] have been produced. We
will show that for both the full symmetric group [Theorem
\ref{zzthm3.12}] and the alternating group [Theorem \ref{zzthm4.15}]
the fixed subring is isomorphic to an AS regular algebra $R$ modulo
a central regular sequence of $R$ (what we call a ``classical
complete intersection" in \cite{KKZ4}). Moreover, we generalize some
results for upper bounds on the degrees of algebra generators for
$\knx^G$ [Theorems \ref{zzthm2.5} and \ref{zzthm2.6}] from results
in the commutative case.

One motivation for this study was to consider the theorem of
Kac-Watanabe \cite{KW}, and independently of Gordeev \cite{G1}, that
provides a necessary condition for any finite group, not
necessarily a permutation group, to have the property that $\kpx^G$
is a complete intersection (the condition is that $G$ be a group
generated by so-called ``bireflections"). This theorem of
Kac-Watanabe-Gordeev was a first step toward the (independent)
classification of finite groups $G$, acting linearly as
automorphisms of  $\kpx$, such that $\kpx^G$ is a complete
intersections that was proven by Gordeev \cite{G2} and Nakajima
\cite{N1,N2, NW}. We verify that an analogous result holds for
$\knx$ and subgroups of the symmetric group $\mathfrak{S}_n$ for $n
\leq 4$ [Example \ref{xxex5.6}], and conjecture that this result is
true in general. We prove that the converse of the
Kac-Watanabe-Gordeev Theorem holds for $\knx$: if $G$ is a group of
permutations of the $\{x_i\}_{i=1}^n$ that is generated by
quasi-bireflections then $\knx^G$ is a classical complete
intersection [Theorem \ref{xxthm5.4}] (this result is not true for
the commutative polynomial ring $\kpx$ [Example \ref{xxex5.5}]).

These fixed rings of $\knx$ under permutation subgroups produce
a tractable class of AS Gorenstein domains that possess a variety of
properties; in many cases their generators have combinatorial
descriptions and their Hilbert series can be described explicitly.
The following table summarizes results presented in this
paper and gives a comparison between the results
of $\knx^G$ with that of $\kpx^G$ for any subgroup
$\{1\}\neq G\subset {\mathfrak{S}_n}$:

\medskip

\begin{tabular}{|c|c|c|}\hline
Statements about $A^G$ &\rule{.1in}{0in} when
$A=\kpx$\rule{.1in}{0in} &
\rule{.1in}{0in} when $A=\knx$ \rule{.1in}{0in}\rule[-.15in]{0in}{.378in}\\
\hline Being AS Gorenstein & Not always & Always \rule[-.15in]{0in}{.378in}\\
\hline Being AS regular  & Sometimes & Never \rule[-.15in]{0in}{.378in}\\
\hline $cci^+(A^{\mathfrak{S}_n})$ & 0 & $\lfloor \frac{n}{2}\rfloor
$
\rule[-.15in]{0in}{.378in}\\
\hline $\deg H_{A^G}(t)$ & $\leq -n$ & $-n$ \rule[-.15in]{0in}{.378in}\\
\hline {\small Bound for degrees of generators} & $\max\{n, {n
\choose 2}\}$ & ${n \choose 2}+\lfloor \frac{n}{2}\rfloor (\lfloor
\frac{n}{2}\rfloor+1)$
\rule[-.15in]{0in}{.378in}\\
\hline  KWG theorem holds & Yes & Conjecture \rule[-.15in]{0in}{.378in}\\
\hline Converse of KWG holds & No & Yes \rule[-.15in]{0in}{.378in}\\
\hline
\end{tabular}
\medskip

\noindent
where KWG stands for Kac-Watanabe-Gordeev.

\section{Definitions and basic properties}
\label{xxsec1}

An algebra $A$ is called {\it connected graded} if
$$A=k\oplus A_1\oplus A_2\oplus \cdots$$
and $A_iA_j\subset A_{i+j}$ for all $i,j\in {\mathbb N}$.
The Hilbert series of $A$ is defined to be
$$H_A(t)=\sum_{i\in {\mathbb N}} (\dim A_i)t^i.$$

\begin{definition}
\label{xxdef1.1}
Let $A$ be a connected graded algebra.
\begin{enumerate}
\item[(1)]
We call $A$ {\it Artin-Schelter Gorenstein} (or {\it AS
Gorenstein}, for short) if the following conditions hold:
\begin{enumerate}
\item[(a)]
$A$ has injective dimension $d<\infty$ on
the left and on the right,
\item[(b)]
$\Ext^i_A(_Ak,_AA)=\Ext^i_{A}(k_A,A_A)=0$ for all
$i\neq d$, and
\item[(c)]
$\Ext^d_A(_Ak,_AA)\cong \Ext^d_{A}(k_A,A_A)\cong k(\bfl)$ for some
integer $\bfl$. Here $\bfl$ is called the {\it AS index} of $A$.
\end{enumerate}
If in addition,
\begin{enumerate}
\item[(d)]
$A$ has finite global dimension, and
\item[(e)]
$A$ has finite Gelfand-Kirillov dimension,
\end{enumerate}
then $A$ is called {\it Artin-Schelter regular} (or {\it AS
regular}, for short) of dimension $d$.
\item[(2)]
If $A$ is a noetherian, AS regular graded domain of global
dimension $n$ and $H_A(t)=(1-t)^{-n}$,
then we call $A$ {\it a quantum polynomial ring}
of dimension $n$.
\end{enumerate}
\end{definition}

Skew polynomial rings $\kqx$, where $q\in k^\times:=k\setminus
\{0\}$, with $\deg x_i=1$ are quantum polynomial rings and also
Koszul algebras. Next we recall from \cite{KKZ1} the definition of a
noncommutative version of a reflection. If $A$ is a connected graded
algebra, let $\Aut(A)$ denote the group of all graded algebra
automorphisms of $A$. If $g\in \Aut(A)$, then the trace function of
$g$ is defined to be
$${\rm Tr}_A(g,t)=\sum_{i=0}^{\infty} \tr(g|_{A_i}) t^i
\in k[[t]],$$
where $\tr(g|_{A_i})$ is the trace of the linear map
$g|_{A_i}$.
Note that ${\rm Tr}_A(g,0)=1$ and that the trace of the identity map
is the Hilbert series of the algebra $A$. The trace of a graded
algebra automorphism of a Koszul algebra can be computed from the
Koszul dual using the following result.

\begin{lemma}
\label{xxlem1.2}
\cite[Corollary 4.4]{JiZ}
Let $A$ be a Koszul algebra with Koszul dual algebra $A^!$.  Let
$g\in \Aut(A)$ and $g^\tau$ be the induced dual automorphism of
$A^!$. Then
$${\rm Tr}_A(g, t) = ({\rm Tr}_{A^!}(g^\tau, -t))^{-1}.$$
\end{lemma}

\begin{definition}
\label{xxdef1.3}
Let $A$ be an AS regular algebra such that
$$H_A(t) =\frac{1}{(1-t)^{n}f(t)}$$
where $f(1)\neq 0$. Let $g\in \Aut(A)$.
\begin{enumerate}
\item
\cite[Definition 2.2]{KKZ1}
Then $g$ is called a {\it quasi-reflection} of $A$ if
$${\rm Tr}_A(g,t)={\frac{1}{(1-t)^{n-1}q(t)}}$$
for $q(1) \neq 0$. If $A$ is a quantum polynomial ring, then
$H_A(t)=(1-t)^{-n}$. In this case $g$ is a {\it quasi-reflection} if
and only if
\begin{equation}\label{E1.3.1}\tag{E1.3.1}
{\rm Tr}_A(g,t)={\frac{1}{(1-t)^{n-1}(1-\lambda t)}} \end{equation}
for some scalar $\lambda \neq 1$. Note that we have chosen not to
call the identity map a quasi-reflection.
\item
\cite[Definition 3.6(b)]{KKZ4}
Then $g$ is called a {\it quasi-bireflection} of $A$ if
$${\rm Tr}_A(g,t)={\frac{1}{(1-t)^{n-2}q(t)}}$$
for $q(1) \neq 0$.
\end{enumerate}
\end{definition}

When $A$ is noetherian and AS Gorenstein and $g$ is in $\Aut(A)$,
the homological determinant of $g$, denoted by $\hdet g$, is defined
in \cite[Definition 2.3]{JoZ}. When $A=\kpx$, the homological
determinant of $g$ is the inverse of determinant of the linear map,
induced by $g$ on the degree one piece $A_1=\bigoplus_{i=1}^n kx_i$
of $A$, and, more generally, it is defined using a scalar map induced
on the local cohomology of $A$; see \cite{JoZ} for details. The
homological determinant is a group homomorphism
$$\hdet: \quad \Aut(A) \rightarrow k^\times.$$
When $A$ is AS regular, the conditions of the following theorem are
satisfied by \cite[Proposition 3.3]{JiZ} and \cite[Proposition 5.5]{JoZ},
and $\hdet g$ can be computed from the trace of $g$, as given in the
following result.

\begin{lemma}
\label{xxlem1.4}
\cite[Lemma 2.6]{JoZ}
Let $A$ be noetherian and AS Gorenstein and let $g\in \Aut(A)$.
If $g$ is $k$-rational in the sense of \cite[Definition 1.3]{JoZ},
then the rational function ${\rm Tr}_A(g,t)$ has the form
$${\rm Tr}_A(g,t) = (-1)^n (\hdet g)^{-1} t^{-\ell}
+ {\rm lower~terms}$$
when it is written as a Laurent series in $t^{-1}$.
\end{lemma}

The following result is not hard to prove.

\begin{theorem}
\label{xxthm1.5} Let $G$ be any subgroup of the symmetric group
$\mathfrak{S}_n$ acting on $\knx$ as permutations.
\begin{enumerate}
\item
The fixed subring $\knx^G$ is an AS Gorenstein domain.
\item
If $G\neq \{1\}$, then $\knx^G$ is not AS regular.
\end{enumerate}
\end{theorem}

\begin{proof} (1)
The trace of any transposition $g=(i,j)$ in $\mathfrak{S}_n$
can be computed using the Koszul dual $(\knx)^!$ by Lemma
\ref{xxlem1.2}, which is isomorphic to
$$k[x_1, \ldots, x_n]/(x_1^2, \ldots, x_n^2 ),$$ and found to be
$${\rm Tr}_A(g,t) = \frac{1}{(1+t^2)(1-t)^{n-2}}
= (-1)^{n-2} t^{-n} + {\rm lower~terms}.$$
It follows from Lemma \ref{xxlem1.4} that the homological determinant
of $g$ is 1. Since $\mathfrak{S}_n$ is generated by transpositions,
$\hdet g=1$ for all $g\in \mathfrak{S}_n$.

By the last paragraph, $\hdet g=1$ for all $g\in G$. The
assertion follows from \cite[Theorem 3.3]{JoZ}.

(2) Since $\knx$ is a quantum polynomial ring, any quasi-reflection
$g$ has the homological determinant $\lambda\neq 1$ where $\lambda$
is given in \eqref{E1.3.1}. Since $\hdet g=1$ for all $g\in G$, $G$
contains no quasi-reflection (also see Lemma \ref{xxlem1.7}(4) below).
The assertion follows from \cite[Theorem 1.1]{KKZ2}.
\end{proof}

The analogous theorem is not true in the commutative case. As we
mentioned in the introduction, $\kpx^{\mathfrak{S}_n}$ is isomorphic
to the commutative polynomial ring $\kpx$, which is AS regular.
Hence Theorem \ref{xxthm1.5}(2) fails for $\kpx$. The next example
shows that Theorem \ref{xxthm1.5}(1) fails for $\kpx$.

\begin{example}
\label{xxex1.6} Set $n=4$. Let $G= \langle (1,2,3,4) \rangle$ be the
cyclic subgroup of $\mathfrak{S}_4$ generated by the 4-cycle
$(1,2,3,4)$. Then $G$ contains no reflections, and has elements of
determinant $-1$, so $B_{+}:=k[x_1, x_2, x_3, x_4]^G$ cannot be
Gorenstein by \cite{Wa}. Or, one can also use Molien's Theorem to
check that the Hilbert series of the fixed subring $B_{+}$ is
$$\frac{t^3+t^2-t +1}{(1-t)^4 (1+t)^2(1+t^2)},$$
which does not have the symmetry property required in Stanley's
criteria \cite[Theorem 4.4]{S2} for $B_{+}$ to be Gorenstein.

Note that for the same subgroup $G$, but acting on the
noncommutative ring $k_{-1}[x_1, \ldots, x_4]$, the Hilbert series
of the fixed ring $B_{-}:=k_{-1}[x_1, \ldots, x_4]^G$ is
$$\frac{(t^2-t+1) (t^6-2t^5+3t^4-2t^3+3t^2-2t+1)}{(1-t)^4(1+t^2)^2(1+t^4)},$$
which has the symmetry property, and hence is AS Gorenstein by a
noncommutative version of Stanley's criteria \cite[Theorem
6.2]{JoZ}, as well as by Theorem \ref{xxthm1.5}(1).
\end{example}

The trace of any permutation is computed as follows.

\begin{lemma}
\label{xxlem1.7} Let $\mathfrak{S}_n$ act on $A= \knx$ as
permutations and $g\in\mathfrak{S}_n$.
\begin{enumerate}
\item
If $g$ is an $m$-cycle, then
$${\rm Tr}_A(g,t) = \frac{1}{(1+ (-t)^m)(1-t)^{n-m}}.$$
\item
If $g = \nu_{i_1} \cdots \nu_{i_k}\mu_{j_1} \cdots \mu_{j_\ell}$ a
product of disjoint cycles of length $i_p$ and $j_p$, with
$\nu_{i_p}$ odd permutations and $\mu_{j_p}$ even permutations, then
$${\rm Tr}_A(g,t) = \frac{1}{(1+t^{i_1})\cdots(1+t^{i_k}) (1-t^{j_1})
\cdots (1-t^{j_\ell}) (1-t)^{n-({i_1}+ \cdots +{i_k}+ j_1 +
\cdots + j_\ell) } }.
$$
\item
The only quasi-bireflections of $\knx$ in $ \mathfrak{S}_n$ are the
two-cycles and three-cycles.
\item
Permutation groups (namely, subgroups of $\mathfrak{S}_n$)
contain no quasi-reflections.
\end{enumerate}
\end{lemma}

\begin{proof} (1) This follows from Lemma \ref{xxlem1.2} and
direct computations.

(2) This follows from part (1) and a graded vector space
decomposition of $\knx$.

(3,4) These are consequences of part (2).
\end{proof}

In \cite{KKZ4} we introduced several possible generalizations of a commutative
complete intersection.  We review these notions here.

\begin{definition}
\label{xxdef1.8}
Let $A$ be a connected graded noetherian algebra.
\begin{enumerate}
\item
We say $A$ is a {\it classical complete intersection} (or a {\it
cci}) if there is a connected graded noetherian AS regular algebra
$R$ and a sequence of regular normal homogeneous elements
$\{\Omega_1,\ldots,\Omega_n\}$ of positive degree such that $A$ is isomorphic
to $R/(\Omega_1,\ldots,\Omega_n)$. The minimum such $n$ is called the {\it
cci-number} of $A$ and denoted by $cci(A)$.
\item
We say $A$ is a {\it hypersurface} if $cci(A)\leq 1$.
\item
We say $A$ is a {\it complete intersection of noetherian type} (or
an {\it nci}) if the $\Ext$-algebra
$\Ext^*_{A}(k,k):=\bigoplus_{i\geq 0} \Ext^i_A(_Ak,_Ak)$ is
noetherian.
\item
We say $A$ is a {\it complete intersection of growth type} (or a
{\it gci}) if the $\Ext$-algebra $\Ext^*_{A}(k,k)$ has finite
Gelfand-Kirillov dimension.
\item
We say $A$ is a {\it weak complete intersection} (or a {\it wci}) if
the $\Ext$-algebra $\Ext^*_{A}(k,k)$ has subexponential growth.
\end{enumerate}
\end{definition}

In \cite{KKZ4} we showed that a property of all of these kinds
of complete intersections is the cyclotomic Gorenstein property
defined below.

\begin{definition}
\label{xxdef1.9}
Let $A$ be a connected graded noetherian algebra.
\begin{enumerate}
\item
We say $A$ is {\it cyclotomic} if its Hilbert series $H_A(t)$ is a
rational function $p(t)/q(t)$ for some coprime polynomials
$p(t),q(t)\in {\mathbb Z}(t)$ and the roots of $p(t)$ and $q(t)$ are
roots of unity.
\item
We say $A$ is {\it cyclotomic Gorenstein} if the following
conditions hold
\begin{enumerate}
\item[(i)]
$A$ is AS Gorenstein;
\item[(ii)]
$A$ is cyclotomic.
\end{enumerate}
\end{enumerate}
\end{definition}

\begin{theorem}
\label{xxthm1.10} \cite[Theorem 3.4]{KKZ4} Let $A$ be $R^G$ for some
noetherian Auslander regular algebra $R$ and a finite subgroup
$G\subset \Aut(R)$. If $A$ is any of the kinds of complete
intersection in Definition \ref{xxdef1.8}, then it is cyclotomic
Gorenstein.
\end{theorem}

We note that in Example \ref{xxex1.6} although the fixed subring
$A^G$ is AS Gorenstein, it is not any of the kinds of  generalized
``complete intersection" of Definition \ref{xxdef1.8} since its
Hilbert series has zeros that are not roots of unity.

The following theorem of Kac-Watanabe-Gordeev is one of the motivations
for this paper.

\begin{theorem}
\label{xxthm1.11} \cite{KW,G1} Let $G$ be a finite group acting
linearly on $\kpx$. If $\kpx^G$ is a complete intersection, then $G$
is generated by bireflections.
\end{theorem}

A noncommutative version of Kac-Watanabe-Gordeev Theorem holds for
skew polynomial rings $\kqx$ when $q\neq \pm 1$ \cite[Theorem
0.3]{KKZ4}, that leaves $\knx$ the only unknown case. In this paper
we will prove some partial results for this special skew polynomial
ring. We note that in Example \ref{xxex1.6} the trace of a
four-cycle acting on  $k_{-1}[x_1, \ldots, x_4]$ is $1/(1+t^4)$,
which is not a quasi-bireflection, supporting a generalization of
the Kac-Watanabe-Gordeev Theorem.

To conclude this section we compute the automorphism group
$\Aut(\knx)$.

\begin{lemma}
\label{xxlem1.12}
\begin{enumerate}
\item
$g\in \Aut(\knx)$ if and
only if $g(x_i)=a_i x_{\sigma(i)}$ for some $\sigma\in {\mathfrak
S}_n$ and $\{a_i\}_{i=1}^n \subset k^\times$, namely,
$\Aut(\knx)=(k^\times)^n \rtimes {\mathfrak S}_n$.
\item
If $g$ is of the form in part (a), then $\hdet g=\prod_{i=1}^{n} a_i$.
\end{enumerate}
\end{lemma}

\begin{proof} (a) Every diagonal map $g: x_i\to a_i x_i$,
for $(a_1,\cdots,a_n)\in (k^\times)^n$, extends easily to a unique
graded algebra automorphism of $\knx$. And we have already seen that
${\mathfrak S}_n$ is a subgroup of $\Aut(\knx)$ such that
${\mathfrak S}_n\cap (k^\times)^n=\{1\}$. Thus $(k^\times)^n \rtimes
{\mathfrak S}_n\subset \Aut(\knx)$. By \cite[Lemma 3.5(e)]{KKZ2},
$\Aut(\knx) \subset (k^\times)^n \rtimes {\mathfrak S}_n$. The
assertion follows.

(b) If $g\in (k^\times)^n$, or $g(x_i)\to a_i x_i$ for
$(a_1,\cdots,a_n)\in (k^\times)^n$, then it is easy to see that
$\hdet g=\prod_{i=1}^n a_i$. If $g\in {\mathfrak S}_n$, then $\hdet
g=1$ by the proof of Theorem \ref{xxthm1.5}(a). The assertion
follows by the fact $\hdet$ is a group homomorphism.
\end{proof}

\section{Upper bound for the algebra generators}
\label{zzsec2}

In this section we show that Broer's and G\"{o}bel's upper bounds on
the degrees of minimal generating sets of $\kpx^G$, for arbitrary
subgroup $G\subset\mathfrak{S}_n$, have analogues in this context.
In this section we do not assume that ${\rm {char}}\; k=0$.

The Noether upper bound on the degrees of generators does not hold
for $\knx$, as $k_{-1}[x_1,x_2]^{\mathfrak{S}_2}$ requires a
generator of degree 3 [Example \ref{zzex3.1}]. More generally one
can ask if the degrees of generators of $\knx^G$ are bounded above
by $|G|$ times the dimension of the representation of $G$. Broer's
degree bound \cite[Proposition 3.8.5]{DK} states that when $f_i$ are
primary invariants, i.e. $f_i$, for $1 \leq i \leq n$, are
algebraically independent and $\kpx^G$ is a finite module over
$k[f_1, \ldots, f_n]$, then $\kpx^G$ is generated as an algebra by
homogeneous invariants of degrees at most
$$\text{deg}(f_1) + \cdots + \text{deg}(f_n) - n.$$
(The above statement is not true when $n=2$ and
$g: x_1\to x_1, x_2\to -x_2$. In this case $f_1=x_1, f_2=x_2^2$.
Therefore we need to assume $n\geq 3$.)
We show that this result
generalizes for any group $G$ (not necessarily a permutation group)
when the given hypotheses are satisfied [Lemma \ref{zzlem2.2}].


Let $A$ be any connected graded algebra. Define $d_A$ to be the
maximal degree of $A_{\geq 1}/(A_{\geq 1})^2$. Then $A$ is generated
as an algebra by homogeneous elements of degree at most $d_A$.

\begin{lemma}
\label{zzlem2.1} Let $A$ be a noetherian connected graded AS
Gorenstein algebra and $B$ and $C$ be graded subalgebras of $A$ such
that $C\subset B\subset A$. Assume that
\begin{enumerate}
\item[(i)]
$A=B\oplus D$ as a right graded $B$-modules,
\item[(ii)]
$A$ is a finitely generated right $C$-module, and
\item[(iii)]
There is a noetherian AS regular algebra $R$ and a surjective graded
algebra map $\phi: R\to C$ and $\gldim R=\injdim A$.
\end{enumerate}
Then
\begin{enumerate}
\item
$\phi$ is an isomorphism and $A_C$ is free.
\item
$d_B\leq \max\{d_C, \bfl_C-\bfl_A\}$ where $\bfl_A$ and $\bfl_C$ are
AS index of $A$ and $C$ respectively.
\end{enumerate}
\end{lemma}

\begin{proof} (1) Let $n=\injdim A$. Induced by the composite map
$f: R\to C \to A$ we have a convergent spectral sequence \cite[Lemma
4.1]{WZ},
$$\Ext^p_A(\Tor^R_q(A, k), A)\Longrightarrow \Ext^{p+q}_R(k,A).$$
Since $A_R$ is finitely generated and $R$ is right noetherian,
$\Tor^R_q(A,k)$ is finite dimensional for all $q$. Thus
$\Ext^p_A(\Tor^R_q(A, k), A)=0$ for all $p\neq \injdim A=n$. The
above spectral sequence collapses to the following isomorphisms
$$\Ext^n_A(\Tor^R_q(A, k), A)\cong \Ext^{n+q}_R(k,A).$$
For any $q>0$, $\Ext^n_A(\Tor^R_q(A, k), A)\cong
\Ext^{n+q}_R(k,A)=0$. Since $\Tor^R_q(A, k)$ is finite dimensional,
$A$ is AS Gorenstein, we obtain that
$$\dim \Tor^R_q(A, k)=\dim \Ext^n_A(\Tor^R_q(A, k), A)=0$$
for all $q>0$. Hence $A_R$ is projective, whence free, as $R$ is
connected graded. As a consequence, $f: R\to A$ is injective. This
implies that $\phi$ is an isomorphism. Since $\phi$ is an
isomorphism and $A_R$ is free, $A_C$ is free.

(2) Now we identify $R$ with $C$. By part (1), $A$ is a finitely
generated free $C$-module. Since $A=B\oplus D$, both $B$ and $D$ are
projective, whence free, graded right $C$-modules. Pick a $C$-basis
for $B$ and $D$, say $V_B\subset B$ and $V_D\subset D$. Then we have
$B=V_B\otimes C$ and $D=V_D\otimes C$. Therefore $A=V_A\otimes C$
where $V_A=V_B\oplus V_D$. Hence
$$H_A(t)=H_{V_{A}}(t)H_C(t)=(H_{V_B}(t)+H_{V_{D}}(t))H_C(t),
\quad {\text{and}}\quad H_B(t)=H_{V_B}(t)H_C(t).$$ Since
$B=V_B\otimes C$, $B$ is generated by $V_B$ and $C$ as a graded
algebra. Thus we have
$$d_{B}\leq \max\{ \deg H_{V_B}(t), d_C\}\leq \max\{ \deg H_{V_A}(t),
d_C\}.$$ It remains to show that $\deg H_{V_A}(t)=-\bfl_A+\bfl_C$.
First, as $H_A(t)=H_{V_A}(t) H_C(t)$, we have $\deg H_{V_A}(t)= \deg
H_A(t)-\deg H_C(t)$. Recall that $C$ is noetherian and AS regular.
Since $A$ is a finite module over $C$, $H_A(t)$ is rational and the
hypotheses (1$^\circ$, 2$^\circ$, 3$^\circ$) of  \cite[Theorem
6.1]{JoZ} hold. By the proof of \cite[Theorem 6.1]{JoZ} (we are not
using the hypothesis that $A$ is a domain),
$$H_A(t)=\pm t^{\bfl_A} H_A(t^{-1})$$
where $\bfl$ is the AS index of $A$. Since $H_A(t)$ is a rational
function such that $H_A(0)=1$, the above equation forces that
\begin{equation}\label{E2.1.1}\tag{E2.1.1}
\deg H_A(t)=-\bfl_A.
\end{equation}
Similarly, $\deg H_C(t)=-\bfl_C$. The assertion follows.
\end{proof}

The degree of algebra generators of $B$ is bounded by
$\bfl_C-\bfl_A$ when $d_C\leq \bfl_C-\bfl_A$, which is easy to
achieve in many cases. The following lemma is a generalization of
Broer's upper bound \cite[Proposition 3.8.5]{DK}.

\begin{lemma}[Broer's Bound]
\label{zzlem2.2} Let $A$ be a quantum polynomial algebra of
dimension $n$ and $C$ an iterated Ore extension
$k[f_1][f_2;\tau_2,\delta_2]\cdots [f_n;\tau_n,\delta_n]$. Assume
that
\begin{enumerate}
\item
$B=A^{H}$ where $H$ is a semisimple Hopf algebra acting on $A$,
\item
$C\subset B\subset A$ and $A_C$ is finitely generated, and
\item
$\deg f_{i}>1$ for at least two distinct $i$'s.
\end{enumerate}
 Then
$$d_{A^H}\leq \bfl_C-\bfl_A=\sum_{i=1}^n \deg f_i -n.$$
\end{lemma}

\begin{proof} Since $H$ is semisimple, $A=B\oplus D$ by
\cite[Lemma 2.4(a)]{KKZ3} where $B=A^H$. Let $R=C$. Then the
hypotheses Lemma \ref{zzlem2.1}(i,ii,iii) hold. By Lemma
\ref{zzlem2.1},
$$d_B\leq \max\{ d_C, \bfl_C-\bfl_A\}.$$

It is clear that $\bfl_A=n$. By induction on $n$, one sees that
$H_C(t)=\frac{1}{\prod_{i=1}^n (1-t^{\deg f_i})}$. By
\eqref{E2.1.1}, $\bfl_C=-\deg H_{C}(t)= \sum_{i=1}^n \deg f_i$. Now
it suffices to show that $d_C\leq \sum_{i=1}^n \deg f_i -n$. For the
argument sake let us assume that $\deg f_i$ is increasing as $i$
goes up. So $d_C=\deg f_n$. Now
$$\sum_{i=1}^n \deg f_i -n=\sum_{i=1}^n (\deg f_i-1)
\geq \deg f_{n-1}-1+\deg f_n-1\geq \deg f_n.$$ The assertion
follows.
\end{proof}

This result applies to subgroups $G\subset \mathfrak{S}_n$ acting on
$\knx$.

Let $C$ be any commutative algebra over $k$ and let $n$ be a
positive integer. Define $D$ be the algebra generated by $C$ and
$\{y_1,\cdots,y_n\}$ subject to the relations
\begin{equation}
\label{E2.2.1}\tag{E2.2.1} [y_i,c]=0
\end{equation}
for all $c\in C$, and
\begin{equation}
\label{E2.2.2}\tag{E2.2.2} y_iy_j+y_jy_i=c_{ij}
\end{equation}
for $1\leq i<j\leq n$, where $\{c_{ij} \mid 1\leq i<j\leq n\}$ is a
subset of the subalgebra $C[y_1^2,\cdots,y_n^2]$ (which is in the
center of $D$).

\begin{lemma}
\label{zzlem2.3} Retain the above notation. Then
\begin{enumerate}
\item
$\sigma: \begin{cases} y_i
\mapsto -y_i & \forall \; i\\
\; c \; \mapsto c & \forall \; c\in C\end{cases}$ extends uniquely
to an algebra automorphism of $D$, and
\item
Let $\{w_1,\cdots,w_n\}$ be a subset of $C[y_1^2,\cdots,y_n^2]$.
Then $\phi:
\begin{cases} y_i
\mapsto w_i & \forall \; i\\
\; c \; \mapsto 0 & \forall \; c\in C\end{cases}$ extends uniquely
to a $\sigma$-derivation of $D$.
\end{enumerate}
\end{lemma}

\begin{proof} (a) Since $D$ is generated by $C$ and $\{y_i\}_{i=1}^n$,
the extension of $\sigma$ is unique. It is clear  that the
extension of $\sigma$ preserves relations \eqref{E2.2.1} and
\eqref{E2.2.2}.

(b)  Since $D$ is generated by $C$ and $\{y_i\}_{i=1}^n$, the
extension of $\phi$, using the $\sigma$-derivation rule, is unique.
For any $c\in C$, using the fact $\phi(c)=0$, we have
$$\quad \phi([y_i,c])=\phi(y_i) c-\sigma(c) \phi(y_i)=w_i c-c w_i=0.$$
For any $i$,
$$\delta(y_i^2)=\sigma(y_i)\delta(y_i)+\delta(y_i)y_i
=-y_i\delta(y_i)+\delta(y_i)y_i=0.$$ As a consequence,
$\delta(c_{ij})=0$. Now
$$\begin{aligned}
\phi(y_iy_j+y_jy_i-c_{ij})&=\phi(y_i)y_j+\sigma(y_i)\phi(y_j)
+\phi(y_j)y_i+\sigma(y_j)\phi(y_i)\\
&=w_i y_j-y_i w_j+w_jy_i-y_j w_i=0.
\end{aligned}
$$
So the extension of $\phi$ is a $\sigma$-derivation.
\end{proof}

We need a lemma on symmetric functions of $\knx$. For every positive
integer $u$, let $P_u$ denote the $u$th power sum $\sum_{i=1}^n
x_i^{u}\in \knx$. Let $C_1$ be the subalgebra of $\knx$ generated by
$P_{2}, P_{4}, \cdots, P_{2n-2}, P_{2n}$, $C_3$ be the subalgebra of
$\knx$ generated by $P_{1}, P_{2}, P_{3}, \cdots, P_{2n-1}, P_{2n}$.
Define $P'_{i}=P_{i}$ is $i$ is odd and $P'_{i}=P_{2i}$ if $i$ is
even. Let $C_2$ be the subalgebra of $\knx$ generated by
$P'_{1},P'_{2}, \cdots, P'_{n-1}, P'_{n}$. Note that
$C_1$ contains $P_{2i}$ for all $i$.

\begin{lemma}
\label{zzlem2.4} Retain the above notation.
\begin{enumerate}
\item
$\knx$ is a finitely generated free module over the central
subalgebra $C_1$.
\item
If $u$ is even, then $P_u P_v=P_vP_u$ for any $v$.
\item
If $u$ and $v$ are odd, then $P_u P_v +P_v P_u=2P_{u+v}$.
\item
If $u$ is odd, then $P_u^2=P_{2u}$.
\item $C_1\subset C_2\subset C_3\subset
\knx^{\mathfrak{S}_n}\subset\knx^G$.
\item
$C_2$ is isomorphic to an iterated Ore extension
$$R:=k[P_{4},P_{8},\cdots, P_{4 \lfloor \frac{n}{2}\rfloor}]
[P_{1}][P_{3}; \tau_{3},\delta_3] \cdots [P_{n'};
\tau_{n'},\delta_{n'}]$$ where $n'=2\lfloor\frac{n-1}{2}\rfloor+1$.
\end{enumerate}
\end{lemma}

\begin{proof} (1) The algebra $\knx$ is a finitely generated module over
$k[x_1^2,\cdots, x_n^2]$ and $k[x_1^2,\cdots, x_n^2]$ is finitely
generated over $C_1=k[P_{2},P_{4}, \cdots, P_{2n}]$ where each
$P_{2i}$ is the $i$th power sum of the variables $\{x_1^2,
\cdots,x_n^2\}$. Therefore $\knx$ is finitely generated over $C_1$.
By the proof of Lemma \ref{zzlem2.1}(1), $\knx$ is free over $C_1$.

(2,3,4) By direct computations.

(5) If $i$ is odd, $(P'_i)^2=(P_i)^2=P_{2i}$, and if $i$ is even,
$P'_i=P_{2i}$. So $C_1\subset C_2$. The rest is clear.

(6) For odd integers $i<j$, part (3) says that
$$P_{j} P_{i}+P_{i}P_{j} =2 P_{i+j}.$$
We can easily determine the automorphisms $\tau_j$ and derivations
$\delta_j$ by using Lemma \ref{zzlem2.3}. As a consequence, there is
a surjective map $\phi: R\to C_2$. Also $\gldim R=n=\gldim \knx$. By
the proof of Lemma \ref{zzlem2.1}(1), $C\cong R$.
\end{proof}

\begin{theorem}[Broer's Bound for $\knx^G$]
\label{zzthm2.5} Let $G$ be a subgroup of $\mathfrak{S}_n$ acting on
$\knx$ naturally. Suppose $|G|$ does not divides ${\text{char}}\;
k$. Then $$d_{(\knx^G)}\leq \frac{1}{2}n(n-1)+\lfloor
\frac{n}{2}\rfloor (\lfloor \frac{n}{2}\rfloor+1)\sim \frac{3}{4}
n^2.$$
\end{theorem}

\begin{proof} The assertion can be checked directly for $n=1,2$.
Assume now that $n\geq 3$. Let $A:=\knx$ and $C$ be $C_2$ as in
Lemma \ref{zzlem2.4}(6). Then $C$ is a subalgebra of $A^G$ for any
$G\subset \mathfrak{S}_n$. Since $|G|$ does not divides
${\text{char}}\; k$, $H:=kG$ is semisimple. Note that $\deg P_i=i$.
Hence all hypotheses in Lemma \ref{zzlem2.2} are satisfied. By Lemma
\ref{zzlem2.2},
$$d_{A^G}\leq \sum_{i=1} \deg f_i-n =\frac{1}{2}n(n+1)+\lfloor
\frac{n}{2}\rfloor (\lfloor \frac{n}{2}\rfloor+1)
-n=\frac{1}{2}n(n-1)+\lfloor \frac{n}{2}\rfloor (\lfloor
\frac{n}{2}\rfloor+1).$$
\end{proof}

This bound is sharp when $n=2$ [Example \ref{zzex3.1}]. For larger
$n$, we have no examples to show this bound is sharp -- and it
probably is not sharp.

Next we consider a generalization of the G\"{o}bel bound \cite{Go}.
If $G$ is a group of permutation of $\{x_i\}_{i=1}^n$  acting as
automorphisms on $\kpx$ then G\"{o}bel's Theorem states that
$\kpx^G$ is generated by the $n$ symmetric polynomials (or the power
sums) and ``special polynomials".  Let $\mathcal{O}_G(X^I)$
represent the orbit sum of $X^I$ under $G$. ``Special polynomials"
are all $G$-invariants of the form $\mathcal{O}_G(X^I)$,
 where $\lambda(I) = (\lambda_i)$, the
partition associated to $I$ (i.e. arranging the elements of $I$ in
weakly decreasing order), has the properties that the last part of
the partition $\lambda_n = 0$, and $\lambda_{i}-\lambda_{i+1} \leq
1$ for all $i$.  It follows that an upper bound on the degree of a
minimal set of generators of $\kpx^G$ for any $n$-dimensional
permutation representation of $G$ is $\max \{ n, \dbinom{n} {2}\}$.
In this context the G\"{o}bel bound can be a sharp bound, as it is
when the alternating group $\mathfrak{A}_n$ acts on $\kpx$. A
similar idea works for $\knx$, see \cite[Corollary 3.2.4]{CA}. But
we consider a modification of $\mathfrak{S}_n$.

Let $\widehat{\mathfrak{S}_n}$ be the group $\mathfrak{S}_n \rtimes
\{\pm 1\}^{n}$, where $\{\pm 1\}^{n}$ is the subgroup of diagonal
actions $x_i\to a_i x_i$ for all $i$, where $a_i=\pm 1$.

\begin{theorem}[G\"{o}bel's Bound for $\knx^G$]
\label{zzthm2.6} Let $G$ be a subgroup of
$\widehat{\mathfrak{S}_n}$. Then
$$d_{\knx^G}\leq n^2, \quad {\text{and}}\quad
d_{\kpx^G}\leq n^2.$$
\end{theorem}

\begin{proof} Let $A$ be $\knx$ or $\kpx$. Let $C=k[P_{2}, P_{4},
\cdots, P_{2n}]$. Then $A$ is a finitely generated free module over
$C$ such that $C\subset A^G$. By Lemma \ref{zzlem2.2},
$$d_{A^G}\leq \sum_{i} \deg f_i-n=\sum_{i} 2i -n=n(n+1)-n=n^2.$$
\end{proof}

\cite[Corollary 3.2.4]{CA} is a consequence of the above theorems.

\section{Invariants under the full symmetric group $\mathfrak{S}_n$}
\label{zzsec3}

Some results in this and the next section have been proved in
\cite{CA}. We repeat some of the arguments for completeness.

We consider the ring of invariants
$\knx^{\mathfrak{S}_n}$ under the full symmetric group
$\mathfrak{S}_n$. Gauss proved that $\kpx^{\mathfrak{S}_n}$ is
generated by the $n$ elementary symmetric functions $\sigma_k$ for
$1 \leq k \leq n$, each of which is an orbit sum (sum of all the
elements in the $\mathfrak{S}_n$-orbit) of the given monomials.
Recall that, for each $1\leq k\leq n$,
$$\sigma_k (x_1, \ldots, x_n) =
\sum_{i_1 < i_2< \cdots < i_k} x_{i_1}x_{i_2} \cdots x_{i_k}.$$
These $\sigma_k$ are algebraically independent, and hence form a
commutative polynomial ring $k[\sigma_1, \ldots, \sigma_n]$. As
a consequence, $cci(\kpx^{\mathfrak{S}_n})=0$. Another
basis of algebraically independent generators of
$\kpx^{\mathfrak{S}_n}$ is the set of the $n$ power sums
$$P_k  = \sum_{i=1}^nx_i^k$$
for $1 \leq k \leq n$. Hence $n$ is the maximal degree of a set of
minimal generators for the fixed subring $\kpx^{\mathfrak{S}_n}$.

The noncommutative case is different. As we have used in the last
section, $P_k$ can be defined in the algebra $\knx$ in the same way,
which is also an $\mathfrak{S}_n$-invariant. However, considered  as
an element in $\knx$, $\sigma_k$ is not an
$\mathfrak{S}_n$-invariant.

\begin{example}
\label{zzex3.1} Let $A=k_{-1}[x_1,x_2]$ and let $G= \langle g
\rangle = \mathfrak{S}_2$ for $g = (1,2)$. Now the element
$\sigma_2=x_1x_2$ is not invariant, and, moreover, $P_2$ is not a
generator because $P_2 = P_1^2$; it is easy to check that there are
no other invariants of degree 2.  We will show that the invariants
are generated by $P_1=x_1 + x_2$ and $P_3=x_1^3+x_2^3$, or by $S_1=
P_1=x_1+x_2$ and $S_2 = x_1^2x_2 + x_1x_2^2$.  In this example the
maximal degree of a minimal set of generators is $3$ [Theorem
\ref{zzthm2.5}], which is larger than the order of the group $|G|$
(the ``Noether bound" \cite{No} guarantees the maximal degree of
a minimal set of generators of $\kpx^G$ is $\leq |G|$). In the
case of either set of generators, the generators are not
algebraically independent, and the ring of invariants is not
AS regular, but AS Gorenstein [Theorem \ref{xxthm1.5}]; and we
will show that it is a cci in a couple ways. First, it is a
hypersurface in the AS regular algebra $B$ generated
by $x,y$ with relations $xy^2=y^2x$ and $x^2y=yx^2$:
$$A^{\mathfrak{S}_2} \cong \frac{B}{(2x^6-3 x^3 y - 3 y x^3 + 4 y^2)}$$
(where $P_1 \mapsto x$ and $P_3 \mapsto y$).  Second, it is a factor
of the iterated Ore extension $C= k[a,b][x][y; \tau, \delta]$, where
$\tau$ is the automorphism of $k[a,b,x]$ defined by $\tau(a)=a,
\tau(b)=b, \tau(x) = -x$, and $\delta$ is a $\tau$-derivation of
$k[a,b,x]$ defined by $\delta(a)= \delta(b) = 0$ and $\delta(x) =
2b$:
$$A^{\mathfrak{S}_2} \cong  \frac{C}{( x^2 -a, y^2 - c )}.$$
Here $\{x^2-a, y^2 -c \}$ for $c= (3ab-a^3)/2$ is a regular sequence
of central elements of $C$.  In this isomorphism
 $P_2 \mapsto a, P_4 \mapsto b, P_1 \mapsto x, P_3 \mapsto y, $
since we have the relations [Lemma \ref{zzlem2.4}]
$$P_3P_1 + P_1 P_3 = 2 P_4$$
$$P_1^2 = P_2$$
$$P_3^2 = P_6 = P_2P_4 - P_2(P_2^2-P_4)/2.$$
\end{example}

The aim of this section is to prove the analogous result for
arbitrary $n$.  We first repeat the analysis from \cite{CA} and
show that there are two sets of algebra generators
of $\knx^{\mathfrak{S}_n}$: the $n$ odd power sums $P_1, P_3,
\cdots, P_{2n-1}$ and the $n$ elements for $1 \leq k \leq n$:
$$S_k = \sum x_{i_1}^2x_{i_2}^2 \cdots x_{i_{k-1}}^2 x_{i_k}=:
\mathcal{O}_{\mathfrak{S}_n}(x_{1}^2 x_{2}^2 \cdots x_{k-1}^2x_k)$$
where the sum is taken over all distinct $i_1, \ldots, i_k$ with
$i_1 < i_2 < \dots < i_{k-1}$, and $\mathcal{O}_{\mathfrak{S}_n}$
represents the sum of the orbit under the full symmetric group; we
call these elements $S_k$ the ``super-symmetric polynomials"
since they play the role that the symmetric functions play in the
commutative case. Hence the maximal degree of a set of minimal
generators for the full ring of invariants $\knx^{\mathfrak{S}_n}$
is $2n-1$.

Any monomial in $\knx$ can be written as the form $\pm
x_1^{i_1}x_2^{i_2}\cdots x_n^{i_n}$, where the sign is due to the
fact that these $x_i$s are $(-1)$-commutative. Let $I$ denote the
index $(i_k):=(i_1,\cdots,i_n)$ and let $X^I$ denote the monomial
$x_1^{i_1}x_2^{i_2}\cdots x_n^{i_n}$. Throughout let $G$ be a
subgroup of $\mathfrak{S}_n$ unless otherwise stated. Define
$$stab_G(X^I)=\{g\in G \mid g(X^I)=X^I {\text{ in }}
\knx\}.$$ For any permutation $\sigma\in G$, $stab_G(X^I)$ and
$stab_G(x_{\sigma(1)}^{i_1}x_{\sigma(2)}^{i_2}\cdots
x_{\sigma(n)}^{i_n})$ are conjugate to each other. As a consequence,
$|stab_G(X^I)|=|stab_G(x_{\sigma(1)}^{i_1}x_{\sigma(2)}^{i_2}\cdots
x_{\sigma(n)}^{i_n})|$.

\begin{definition}
\label{zzdef3.2} Let $\lambda(m) = (\lambda_1, \lambda_2, \cdots,
\lambda_n)$ be a partition $m$, where $\lambda_i$ are weakly
decreasing and  $\lambda_i \geq 0$. Let $X^\lambda$ be the monomial
$x_1^{\lambda_1} x_2^{\lambda_2} \cdots x_n^{\lambda_n}$. The
$G$-orbit sum of the monomial $X^\lambda$ of (total) degree $m$ is
defined by
$$\mathcal{O}_{G}(X^\lambda)=
\mathcal{O}_{G}(x_1^{\lambda_1} x_2^{\lambda_2} \cdots
x_n^{\lambda_n}) = \frac{1}{|stab_G(X^\lambda)|}\sum_{g \in
G}x_{g(1)}^{\lambda_1} x_{g(2)}^{\lambda_2} \cdots
x_{g(n)}^{\lambda_n}.$$
\end{definition}

In this section we take $G=\mathfrak{S}_n$ and in the next
$G=\mathfrak{A}_n$.

\begin{remark}
\label{zzrem3.3} We divide by the order of the stabilizer of
$X^\lambda$ so that each element of the orbit is counted only once.
Throughout we will compare monomials using the
length-lexicographical order: for $I = (i_k)$ and $J= (j_k)$ we say
$X^I < X^J$ if $\sum i_k < \sum j_k$, or if $\sum i_k = \sum j_k$,
and if $k$ is the smallest index for which $i_k \neq j_k$ then $i_k
< j_k$; when considering elements of the same degree this order is
the lexicographical order on the exponents with $x_1 > x_2 > \ldots
> x_n$. Hence we will denote the $\mathfrak{S}_n$-orbit sum by
$\mathcal{O}_{\mathfrak{S}_n}(X^I)$, where $X^I$ is the leading term
of the orbit sum under the (length)-lexicographic order and so $I$
is a partition, and we call $\mathcal{O}_{\mathfrak{S}_n}(X^I)$ the
$\mathfrak{S}_n$-orbit sum corresponding to the partition $I$. We
refer to the entries in $I$ as the ``parts" of the partition (so a
part may be $0$).
\end{remark}

The following lemma is easily verified.

\begin{lemma}\cite[Theorem 2.1.3]{CA}
\label{zzlem3.4}
Let $G$ be a finite subgroup of $\mathfrak{S}_n$.
Then any $G$-invariant is a sum of homogeneous $G$-invariants and
homogeneous invariants are linear combinations of $G$-orbit sums.
\end{lemma}

\begin{lemma}\cite[Lemma 2.2.2]{CA}
\label{zzlem3.5} A $\mathfrak{S}_n$-orbit sum corresponding to a
partition $\lambda(m) = (\lambda_1, \lambda_2, \cdots, \lambda_n)$
is zero if and only if it has repeated odd parts.  Hence a non-zero
$\mathfrak{S}_n$-orbit sum corresponds to a partition with no
repeated odd parts.
\end{lemma}



\begin{proof}
An orbit sum $\mathcal{O}_{\mathfrak{S}_n}(X^I)$ is zero if and only
if the $\mathfrak{S}_n$-orbit of $X^I$ consists of monomials and
their negatives, i.e. $\sigma X^I = -X^I$ for some $\sigma \in
\mathfrak{S}_n$. In order for $\sigma X^I = -X^I$ there must be a
repeated exponent. Consider a monomial of the form $x_1^{e_1}\cdots
x_j^{e_j} \cdots x_k^{e_k} \cdots x_n^{e_n}$ where $e_j=e_k$ and
both are odd.  We claim that when the transposition $(j,k)$ is
applied to this monomial we get the same monomial but with a
negative sign.  We induct on $k-j$.  If $k-j=1$ then the result is
clear. Hence assume that result is true for $k-j < \ell$ and we
prove it for $k-j = \ell$. We write the monomial as $x_1^{e_1}\cdots
x_j^{e_j} \cdots x_{k-1}^{e_{k-1}} x_k^{e_k} \cdots x_n^{e_n}$ and
consider the case when $e_{k-1}$ is odd and the case when $e_{k-1}$
is even.  When $e_{k-1}$ is odd then $(j,k)$ applied to the monomial
yields
$$\begin{aligned}
x_1^{e_1}&\cdots x_k^{e_k} \cdots x_{k-1}^{e_{k-1}} x_j^{e_j} \cdots
x_n^{e_n}\\
&= -x_1^{e_1}\cdots x_k^{e_k} \cdots  x_j^{e_j}
x_{k-1}^{e_{k-1}}\cdots x_n^{e_n} \end{aligned}$$ which by induction
is
$$\begin{aligned}
\quad\quad &= x_1^{e_1}\cdots x_j^{e_j} \cdots  x_k^{e_k}
x_{k-1}^{e_{k-1}}\cdots x_n^{e_n}\\
&= -x_1^{e_1}\cdots x_j^{e_j} \cdots x_{k-1}^{e_{k-1}} x_k^{e_k}
\cdots x_n^{e_n}. \end{aligned}$$ When $e_{k-1}$ is even then
$(j,k)$ applied to the monomial yields
$$\begin{aligned}
\; x_1^{e_1}& \cdots x_k^{e_k} \cdots x_{k-1}^{e_{k-1}} x_j^{e_j}
\cdots x_n^{e_n}\\
&= x_1^{e_1}\cdots x_k^{e_k} \cdots  x_j^{e_j}
x_{k-1}^{e_{k-1}}\cdots x_n^{e_n}\quad \end{aligned}$$ which by
induction is
$$\begin{aligned}
\quad \quad &= -x_1^{e_1}\cdots x_j^{e_j} \cdots  x_k^{e_k}
x_{k-1}^{e_{k-1}}\cdots x_n^{e_n}\\
&= -x_1^{e_1}\cdots x_j^{e_j} \cdots x_{k-1}^{e_{k-1}} x_k^{e_k}
\cdots x_n^{e_n}.\end{aligned}$$
Hence $\sigma X^I = - X^I$, and so
for any $\tau X^I$ in the $\mathfrak{S}_n$-orbit of $X^I$ we have $-
\tau X^I = \tau \sigma X^I$ is in the orbit of $X^I$, and hence the
$\mathfrak{S}_n$-orbit sum of $X^I$ is zero.

Clearly when indices with even exponents of the same value are
permuted no sign change occurs, and so the orbit sum will not be
zero unless there is at least one repeated odd exponent.
\end{proof}

By Lemma \ref{zzlem3.5} the set of elements in $\knx^{\mathfrak{S}_n}$
of degree $k$ has a vector space basis corresponding to the partitions
of $k$ into at most $n$ parts with no repeated odd entries.  We next
will show that both the sets $S_k$  and $P_{2k-1}$ for $k = 1, \ldots, n$
(corresponding to the partitions $(2,\ldots, 2,1, 0, \ldots, 0)$ and
$(2k-1,0,\ldots,0)$ of $2k-1$, respectively) are algebra generators
of $\knx^{\mathfrak{S}_n}$.

\begin{lemma}
\label{zzlem3.6} Let $I = (\lambda_k)$ be a partition where no
$\lambda_i$ are both equal and odd. The leading term of
$\mathcal{O}_{\mathfrak{S}_n}(x_1^{\lambda_1} x_2^{\lambda_2} \cdots
x_n^{\lambda_n}) S_k$ is $x_1^{\lambda_{1} +2} \cdots
x_{k-1}^{\lambda_{k-1}+2} x_{k}^{\lambda_{k}+1}
x_{k+1}^{\lambda_{k+1}} \cdots x_n^{\lambda_n}$.
\end{lemma}

\begin{proof}
By our assumption on $I$ the orbit of $X^I$ does not contain another
element with the same entries as $X^I$. Clearly $x_1^{\lambda_{1} +2}
\cdots x_{k-1}^{\lambda_{k-1}+2} x_{k}^{\lambda_{k}+1}
x_{k+1}^{\lambda_{k+1}} \cdots x_n^{\lambda_n}$ is a summand of the
product of $\mathcal{O}_{\mathfrak{S}_n}(x_1^{\lambda_1}
x_2^{\lambda_2} \cdots x_n^{\lambda_n}) S_k$.  This product of
orbits can be written as a linear combination of
$\mathfrak{S}_n$-orbit sums; let $\mathcal{O}_{\mathfrak{S}_n}(X^E)$
be one of these orbit sums.  The entries of the any such partition
$E$  are obtained from the partition $I = (\lambda_1, \cdots,
\lambda_n)$ by adding $2$ to $k-1$ entries of $I$, adding $1$ to one
entry of $I$, and placing these entries into numerical order.  It is
clear that the largest such partition $E$ that can be obtained in
this manner is $(\lambda_1 + 2, \ldots, \lambda_{k-1}+2, \lambda_{k}
+1, \lambda_{k+1}, \cdots, \lambda_n)$, and the leading term of this
$\mathfrak{S}_n$-orbit sum occurs in the product of orbits only
once.
\end{proof}

The following lemma follows essentially as in Gauss's proof for
$\kpx^{\mathfrak{S}_n}$; the super-symmetric polynomials $S_k \in
\knx^{\mathfrak{S}_n}$ play the role of the symmetric polynomials
$\sigma_k$ in $\kpx^{\mathfrak{S}_n}$.

\begin{lemma}
\label{zzlem3.7} Suppose that $f \neq 0$ is a
$\mathfrak{S}_n$-invariant with leading term $x_1^{\lambda_1}
x_2^{\lambda_2} \cdots x_n^{\lambda_n}$ of degree $m$ where at least
one $\lambda_k$ odd. Then there is a positive integer $k$, a
partition  $\lambda^*(m-2k+1)=(\lambda^*_1, \ldots, \lambda^*_n)$ of
$m-2k+1$, and  a $c  \in k^\times$ such that
$$f- c \; \mathcal{O}_{\mathfrak{S}_n}(x_1^{\lambda^*_1} \cdots
x_n^{\lambda^*_n}) S_k$$ has leading term of smaller degree than
$f$.  As a consequence, the fixed subring $\knx^{\mathfrak{S}_n}$ is
generated as an algebra by the $n$ elements $S_k$, for $k=1, \ldots,
n$, and invariants with all even powers, $k[x_1^2, \cdots,
x_n^2]^{\mathfrak{S}_n}$.
\end{lemma}

\begin{proof}  $I= (\lambda_i)$ is a partition and hence is
weakly decreasing. Let $k$ be the largest index with $\lambda_k$ odd,
and let $$I^* =
(\lambda_1-2, \lambda_2-2, \ldots, \lambda_{k-1}-2, \lambda_k -1
,\lambda_{k+1}, \ldots, \lambda_n).$$  We claim that $I^*$ is a
weakly decreasing sequence.  First note that since $\lambda_k$ is
odd, $\lambda_k \geq 1$, and for $\ell \geq k+1$ the $\lambda_\ell$
are even and weakly decreasing, so for $\ell \geq k+1$ we have
$\lambda_k \geq \lambda_\ell + 1 \geq \lambda_{\ell+1} +1$, and the
final $n-k+1$ entries of $I^*$ are weakly decreasing. Next, since
$\lambda_k$ is odd and there are no repeated odd exponents in a
nonzero $\mathfrak{S}_n$-orbit sum, we have $\lambda_{k-1} \geq
\lambda_k + 1$ and $\lambda_{j-2}-2 \geq \lambda_{j-1} -2 \geq
\lambda_k -1$ for $3 \leq j \leq k$, so the first $k$ entries of
$I^*$ are weakly decreasing.   Hence by Lemma \ref{zzlem3.6}  we
have
$$x_1^{\lambda_1} x_2^{\lambda_2} \cdots x_n^{\lambda_n}=
(x_1^{\lambda_1-2} \cdots x_{k-1}^{\lambda_{k-1}-2}
x_k^{\lambda_k-1} \cdots x_n^{\lambda_n})(x_1^2 \cdots x_{k-1}^2
x_k)$$ is the leading term in $\mathcal{O}_{\mathfrak{S}_n}(X^{I^*})
S_k$, and if $c$ is the coefficient of the leading term of $f$ then
$c \; \mathcal{O}(X^{I^*}) S_k-f$ has smaller order leading term.
Furthermore $\mathcal{O}(X^{I^*})$ also has smaller order.  Since
there are only a finite number of smaller orders, the algorithm must
terminate when all exponents are even.
\end{proof}



Since the central subring $k[x_{1}^2, \ldots, x_{n}^2]$ of $\knx$ is
a commutative polynomial ring and $\mathfrak{S}_n$ acts on it as
permutations, the invariants $k[x_{1}^2, \ldots,
x_{n}^2]^{\mathfrak{S}_n}$ are generated by either the even power
sums $P_2, \cdots, P_{2n}$ or the $n$ symmetric polynomials in the
squares; in particular, if $\rho_i:=\sigma_i(x_1^2, \cdots, x_n^2)$
for the elementary symmetric function $\sigma_i$, then $k[x_{1}^2,
\cdots, x_{n}^2]^{\mathfrak{S}_n} = k[\rho_1, \rho_2, \ldots ,
\rho_n]$. Since $P_{2k} \in k[\rho_1, \rho_2, \ldots , \rho_n],$
each $P_{2k}$ can be expressed as a polynomial in the elementary
symmetric functions, say
\begin{equation} \label{E3.7.1}\tag{E3.7.1}
P_{2k} = f_{2k}(\rho_1, \rho_2, \ldots , \rho_n).
\end{equation}

Next we show that  $k[x_{1}^2, \ldots, x_{n}^2]^{\mathfrak{S}_n}$ is
contained in the algebra generated by the $n$ odd power sums $P_1,
\ldots, P_{2n-1}$, and $k[x_{1}^2, \ldots,
x_{n}^2]^{\mathfrak{S}_n}$ is contained in the algebra generated by
the $n$ super-symmetric polynomials $S_k$.

\begin{lemma}
\label{zzlem3.8} The fixed subring $k[x_{1}^2, \ldots,
x_{n}^2]^{\mathfrak{S}_n}$ is contained in the algebra generated by
either the odd power sums $P_1, \cdots, P_{2n-1}$ or by the
super-symmetric polynomials $S_1, \cdots, S_n$ in $\knx$.
\end{lemma}

\begin{proof}
We obtain the even power sums from the odd ones as follows: $P_2 =
P_1^2$, and more generally
\begin{equation}
\label{E3.8.1}\tag{E3.8.1} P_{2i} = (P_1 P_{2i-1} + P_{2i-1} P_1 )/2
\end{equation}
for all $1\leq i\leq n$. Also
\begin{equation}
\label{E3.8.2}\tag{E3.8.2} \rho_j=
\mathcal{O}_{\mathfrak{S}_n}(x_1^2 \cdots x_j^2) = (S_1 S_{j} +
S_{j} S_1 )/(2j)
\end{equation}
 for all $1\leq j\leq
n$.
\end{proof}

The next argument follows as in the case of $\kpx$ \cite[p. 4]{S1}.
Given a monomial $X^I$, we define $\lambda(I)$, the partition
associated with $X^I$, to be the elements of $I$ listed in weakly
decreasing order (i.e. the partition associated to
$\mathcal{O}_{\mathfrak{S}_n}(X^I)$). We define a total order on the
set of monomials as $X^I < X^J$ if the associated partitions have
the property that $\lambda(I)$ is lexicographically {\it larger}
than $\lambda(J)$, or, if the partitions are equal, when $I$ is
lexicographically {\it smaller} than $J$.  As an example for $n=3$
and degree $= 4$
$$x_3^4 < x_2^4 < x_1^4< x_2x_3^3 < x_2^3x_3 < x_1 x_3^3 < x_1x_2^3
< x_1^3x_3 < x_1^3x_2 < x_2^2x_3^2$$
$$<x_1^2x_3^2 < x_1^2x_2^2 < x_1x_2x_3^2<x_1x_2^2x_3<x_1^2x_2x_3.$$
In the case of $\kpx$, where all partitions represent basis elements
in the subring of invariants, in a given degree $k \leq n$ the
``largest" partition is $(1, \ldots, 1, 0 \ldots, 0)$, while the
``smallest" partition is $(k,0 \ldots,0)$. In the case of of $\knx$,
for monomials that correspond to nonzero invariants there are no
repeated odd parts, so for odd degrees $2k-1 \leq 2n-1$, the
partition $(2, \ldots, 2, 1, 0, \ldots, 0)$ is ``largest" under this
order, and while the partition $(2k-1,0 \ldots,0)$ is smallest, and
$x_n^{2k-1}$  is the smallest monomial of degree $2k-1$. Furthermore
in a product of power sums
$$P_{i_1}P_{i_2} \cdots P_{i_k}$$
the leading monomial will be $cx_1^{i_1}x_2^{i_2} \cdots x_k^{i_k}$
for some nonzero integer $c$ when the $i_j$ are weakly decreasing.

\begin{lemma}
\label{zzlem3.9} The  fixed subring $\knx^{\mathfrak{S}_n}$ is
generated by the $n$ odd power sums $P_1, \ldots, P_{2n-1}$.
\end{lemma}

\begin{proof} By Lemma \ref{zzlem3.8}
the even power sums are generated by the odd power sums $P_1,
\ldots, P_{2n-1}$, so it suffices to show invariants are generated
by power sums $P_k$ for $k \leq 2n-1$. By Lemmas \ref{zzlem3.7} and
\ref{zzlem3.8} the $S_k$ are algebra generators of
$\knx^{\mathfrak{S}_n}$, so it suffices to show they can be
expressed in terms of power sums. Hence it suffices to describe an
algorithm that writes an invariant $f \in \knx^{\mathfrak{S}_n}$ of
degree $\leq 2n-1$ as a product of power sums.  Write the leading
term of $f$ as $a x_1^{i_1}x_2^{i_2} \cdots x_n^{i_n}$ for some $a
\in k^\times$.  The exponents of the leading term are weakly
decreasing, and each is $\leq 2n-1$.  The element $f - \frac{a}{c}
P_{i_1} P_{i_2} \cdots P_{i_n}$ has the same total degree as $f$,
but its leading term is less than that of $f$.  Since there are only
a finite number of monomials of smaller order for a fixed degree,
the algorithm terminates with $f$ written in terms of power sums of
degree $\leq 2n -1$.
\end{proof}



The following theorem of Cameron Atkins follows from the lemmas above,
and gives us two choices of algebra generators for
$\knx^{\mathfrak{S}_n}$. It is often convenient to choose the power
sums, since they have fewer summands.

\begin{theorem}\cite[Theorems 2.2.6 and 2.2.8]{CA}
\label{zzthm3.10} The fixed subring $\knx^{\mathfrak{S}_n}$ is
generated by either the set of the $n$ odd power sums $P_1, \cdots,
P_{2n-1}$ or the set of the $n$ super-symmetric polynomials $S_1,
\cdots, S_n$.
\end{theorem}

We next show that the AS Gorenstein domain $\knx^{\mathfrak{S}_n}$
is a cci. First we have to construct a
suitable AS regular algebra.

Let $R= k[p_1, p_2, \ldots , p_n]$ be a commutative polynomial ring,
and let $a_{2i} = f_{2i}(p_1, p_2, \ldots , p_n)$ where the $f_{2i}$
are the polynomials of \eqref{E3.7.1}. Consider the following
iterated Ore extension
$$B=k[p_1, \ldots, p_n][y_1: \tau_1, \delta_1] \cdots
[y_n: \tau_n, \delta_n]$$ where coefficients are written on the
left, $R=k[p_1, \ldots, p_n]$ is a commutative polynomial ring,
$\tau_j$ is the automorphism of $k[p_1, \ldots, p_n][y_1: \tau_1,
\delta_1] \cdots [y_{j-1}: \tau_{j-1}, \delta_{j-1}]$ defined by
$\tau_{j}(y_i) = -y_i$ for $i < j$ and $\tau_j(r) = r$ for $r \in
k[p_1, \ldots, p_n]$, and $\delta_{j}$ is the $\tau_j$-derivation
$\delta_{j}(y_i) = 2 a_{2i+2j-2}$ with $\delta_j(r)=0$ for all $r
\in k[p_1, \ldots, p_n]$.

By Lemma \ref{zzlem2.3}, $\delta_k$ are $\tau_k$-derivation for all
$k$ where $(\tau_k \delta_k)$ appeared in the definition of $B$.


We grade $B$ by setting degree$(p_i) =2i$ and degree$(y_i) =2i-1$.
With this grading the Hilbert series of $B$ is given by
\[ H_B(t) = \frac{1}{(1-t)(1-t^2)\cdots (1-t^{2n-1})(1-t^{2n})}. \]
The algebra $B$ is an AS regular algebra of dimension $2n$.
Let $r_i =y_i^2-a_{4i-2}$ for each $i=1,2, \ldots, n$; it is easy to
see that $r_i$ is a central element of $B$.

\begin{lemma} \label{zzlem3.11}
The sequence $\{r_1, r_2, \ldots , r_n\}$ is a central regular
sequence in $B$.
\end{lemma}

\begin{proof}
First we note that the $r_i$ are central since  $a_{i}$  and $y_i^2$
are central
$$y_i^2 y_j = y_i(-y_jy_i + p_{i+j}) = -y_iy_jy_i + y_i
p_{i+j} = -(y_iy_j + p_{i+j})y_i  = y_jy_i^2.$$

Since $B$ is a domain,  $r_1\neq 0$ is regular in $B$.

Let $B_i = k[p_1, \ldots, p_n][y_1: \tau_1, \delta_1] \cdots [y_i:
\tau_i, \delta_i]$ and let $\overline{B_i} = B_i/(r_1, r_2, \ldots ,
r_i)_{B_i}$. Now consider the algebra $C_i = \overline{B_i}[y_{i+1}:
\overline{\tau_{i+1}}, \overline{\delta_{i+1}}] \cdots [y_n:
\overline{\tau_n}, \overline{\delta_n}]$, where the
$\overline{\tau_j}$ and $\overline{\delta_j}$ are the induced maps.
These maps are well-defined since for $j>i$ and $k \leq i,
\tau_j(r_k) = r_k$ and $\delta_j(r_k) = 0$. Note that $B =
B_i[y_{i+1}: \tau_{i+1}, \delta_{i+1}] \cdots [y_n: \tau_n,
\delta_n]$, and hence every element of $B$ can be written in the
form $\sum_I b_I y^I$ where $b_I\in B_i, I = (e_{i+1}, e_{i+2},
\ldots , e_n)$ is a nonnegative integral vector, and $y^I =
y_{i+1}^{e_{i+1}}y_{i+2}^{e_{i+2}}\cdots y_n^{e_n}.$ The algebra
$B/(r_1, r_2, \ldots , r_i)_B$ is isomorphic to the algebra $C_i$
under the map
\[ \sum_I b_Iy^I + \langle r_1, r_2, \ldots ,
r_i\rangle_B \mapsto \sum_I \bar{b_I}y^I \] where $\bar{b_I}$
denotes reduction mod $( r_1, r_2, \ldots , r_i)_{B_i}.$ Now the
standard polynomial degree argument in $C_i$ shows that the image of
$r_{i+1}$ is regular in $C_i$.
\end{proof}

We now can prove that
$\knx^{\mathfrak{S}_n} \cong B/(r_1, r_2, \ldots , r_n)$
where, by Lemma \ref{zzlem3.11},  each $r_i$ is central in $B$ and
regular in $B/(r_1, r_2, \ldots , r_{i-1})$.

\begin{theorem} \label{zzthm3.12}
The algebra $\knx^{\mathfrak{S}_n}$ is a cci.
\end{theorem}
\begin{proof}
By Definition \ref{zzdef3.2} and Lemma \ref{zzlem3.5}
$\knx^{\mathfrak{S}_n}$ as a graded vector space has a basis  of
orbit sums of monomials having no repeated odd exponents.  Hence its
Hilbert series is the same as the generating function for the
restricted partitions having no repeated odd parts.  By Proposition
\ref{xxpro5.1} of the Appendix this Hilbert series is given by
\begin{eqnarray*}
D_n(t)&=& \frac{(1-t^2)(1-t^6)(1-t^{10})\cdots
(1-t^{4n-2})}{(1-t)(1-t^2)(1-t^3) \cdots (1-t^{2n-1})(1-t^{2n})}.
\end{eqnarray*}

Let $\rho_i = \sigma_i(x_1^2, x_2^2, \ldots , x_n^2)$ where
$\sigma_i$ is the $i$th elementary symmetric polynomial. Then the
algebra $k[x_1^2, x_2^2, \ldots , x_n^2]^{\mathfrak{S}_n} =
k[\rho_1, \rho_2, \ldots, \rho_n]$ is a commutative polynomial ring.
By Theorem~\ref{zzthm3.10}, $\knx^{\mathfrak{S}_n}$ is generated as
an algebra by the odd power sums, and hence $\knx^{\mathfrak{S}_n} =
k[\rho_1, \rho_2, \ldots, \rho_n][P_1, P_3, \ldots , P_{2n-1}].$

Consider the iterated Ore extension $B$ constructed above and define
a map $\phi: B \longrightarrow \knx^{\mathfrak{S}_n}$ by $\phi(p_i)
= \rho_i$ and $\phi(y_j) = P_{2j-1}.$ Note that $\phi$ preserves
degree. Clearly $\phi$ takes $R= k[p_1, p_2, \ldots , p_n]$
isomorphically onto  $k[\rho_1, \rho_2, \ldots, \rho_n],$ and  both
subrings are central. In the  iterated Ore extension $B$, we have
for $i < j$ that
\[ y_jy_i + y_iy_j = 2a_{2i+2j-2} =f_{2i+2j-2}(p_1, p_2, \ldots , p_n). \]
Calculation in $\knx^{\mathfrak{S}_n}$ shows that
\[ P_{2j-1}P_{2i-1} + P_{2i-1}P_{2j-1} = 2P_{2i+2j-2} =
2f_{2i+2j-2}(\rho_1, \rho_2, \ldots , \rho_n); \]
hence
\[ \phi(y_j)\phi(y_i) + \phi(y_i)\phi(y_j) = 2\phi(a_{2i+2j-2}). \]
Hence the skew extension relations are preserved, and we conclude
that $\phi$ is a graded ring homomorphism.  Since the odd power sums
$P_1, P_3, \ldots , P_{2n-1}$ generate $\knx^{\mathfrak{S}_n}$ as an
algebra by Theorem~\ref{zzthm3.10}, the homomorphism $\phi$ is an
epimorphism.

Calculation yields
\begin{eqnarray*}
0 &=& P_{2i-1}^2 -P_{4i-2} = P_{2i-1}^2 - f_{4i-2}
(\rho_1, \rho_2, \ldots , \rho_n) \\
&=& \phi(y_i^2 - a_{4i-2}) = \phi(r_i).
\end{eqnarray*}
Hence the ideal $(r_1, r_2, \ldots , r_n) \subseteq \ker(\phi)$, and
$\phi$ induces a graded ring homomorphism
\[ \bar{\phi}: B/(r_1, r_2, \ldots , r_n) \longrightarrow
\knx^{\mathfrak{S}_n} . \]
Since for each $i$ the degree of $r_i$ is $4i-2$ and $\{r_1, r_2,
\ldots ,r_n\}$ is a regular sequence, the Hilbert series of $\bar{B}
=B/(r_1, r_2, \ldots , r_n)$ is given by
\begin{eqnarray*}
 H_{\bar{B}}(t)&=& \frac{(1-t^2)(1-t^6)(1-t^{10})\cdots (1-t^{4n-2})}
 {(1-t)(1-t^2)(1-t^3) \cdots (1-t^{2n-1})(1-t^{2n})}.
 \end{eqnarray*}
This shows that $\bar{\phi}$ is an isomorphism.
\end{proof}

\begin{definition}
\label{zzdef3.13} Let $A$ be a connected graded noetherian algebra.
\begin{enumerate}
\item
We say $A$ is a {\it classical complete intersection$^+$} (or a {\it
cci$^+$}) if there is a connected graded noetherian AS regular
algebra $R$ with $H_R(t)=\frac{1}{\prod_{i=1}^n (1-t^d_i)}$ and a
sequence of regular normal homogeneous elements
$\{\Omega_1,\ldots,\Omega_n\}$ of positive degree such that $A$ is
isomorphic to $R/(\Omega_1,\ldots,\Omega_n)$. The minimum such $n$
is called the {\it cci$^+$-number} of $A$ and denoted by $cci^+(A)$.
\item
Let $A$ be cyclotomic (e.g., $A$ is cci). The $cyc$-number of $A$,
denoted by $cyc(A)$,
is defined to be $v$ if the Hilbert series of $A$ is of the form
$$
H_A(t)=\frac{\prod_{s=1}^v (1-t^{m_s})}{\prod_{s=1}^w (1-t^{n_s})}
$$
where $m_s\neq n_{s'}$ for all $s$ and $s'$.
\end{enumerate}
\end{definition}

Clearly we have $cci^{+}(A)\geq cci(A)$. It is a conjecture that
every noetherian AS regular algebra has Hilbert series of the form
$\frac{1}{\prod_{i=1}^n (1-t^d_i)}$. If this conjecture holds,
then being cci$^+$ is equivalent to being
cci and $cci^{+}(A)= cci(A)$. One can easily show that
the expression of $H_A(t)$ in Definition \ref{zzdef3.13}(2) is unique
(as we assume that $m_s\neq n_{s'}$ for all $s,s'$). It follows from
the definition that $cci^+(A)\geq cyc(A)$.
Finally we would like to calculate $cci^+(\knx^{\mathfrak{S}_n})$.

\begin{theorem}
\label{zzthm3.14} $cci^+(\knx^{\mathfrak{S}_n})
=cyc(\knx^{\mathfrak{S}_n})=\lfloor \frac{n}{2}\rfloor$.
\end{theorem}

\begin{proof} First we prove the claim that $cci^+(\knx^{\mathfrak{S}_n})
\leq \lfloor \frac{n}{2} \rfloor$.

Let $C_2$ be the subalgebra of $\knx^{\mathfrak{S}_n}$
defined before Lemma \ref{zzlem2.4}. By Lemma \ref{zzlem2.4}(6), it
is isomorphic to the iterated Ore extension
$$k[P_4,P_8,\cdots, P_{4 \lfloor \frac{n}{2}\rfloor}]
[P_1][P_3;\tau_{3},\delta_3]\cdots [P_{n'};\tau_{n'},\delta_{n'}]$$
where $n'=2\lfloor \frac{n-1}{2}\rfloor+1$. By Lemma \ref{zzlem2.4}(5),
$C_2$ contains $P_{2i}$ for all $i\geq 1$. Let $F_{n'}:=C_2$, and for
any odd integer $n'< j\leq 2n-1$, we inductively construct a sequence
of iterated Ore extensions $F_{j}=F_{j-2}[P_j, \tau_j, \delta_j]$
where $\tau_{j}$ is defined by $\tau_{j}(P_{s})=(-1)^s P_{s}$
for all even $s$ and for all odd $s\leq j-2$, and where the
$\tau_j$-derivation $\delta_j$ is defined by
$\delta_j(P_s)=\begin{cases} 0& {\text{ if $s$ is even}}\\
2P_{s+j} & {\text{ if $s$ is odd}}.\end{cases}.$ It follows from
the induction and Lemma \ref{zzlem2.3} that $\tau_j$ is an
automorphism of $F_{j-2}$ and $\delta_j$ is a $\tau_j$-derivation of
$F_{j-2}$. Therefore $F_{j}$ (and whence $F_{2n-1}$) is an iterated
Ore extension (which is a noetherian AS regular algebra with Hilbert
sires of the form $(\prod_{i=1}^n (1-t^{d_i}))^{-1}$). Let
$u_s=P_{2s-1}^2-P_{4s-2}$ for all integers from $s=\lfloor
\frac{n-1}{2}\rfloor+2$ to $s=n$. The proof of Lemma \ref{zzlem3.11}
shows that $\{u_{\lfloor \frac{n-1}{2}\rfloor+2}, \cdots, u_{n}\}$
is a central regular sequence of $F_{2n-1}$. It is easy to see that
$F_{2n-1}/(u_{\lfloor \frac{n-1}{2}\rfloor+2}, \cdots, u_{n})\cong
\knx^{\mathfrak{S}_n}$. Therefore $cci^{+}(\knx^{\mathfrak{S}_n})\leq
n-(\lfloor \frac{n-1}{2}\rfloor+1)= \lfloor \frac{n}{2}\rfloor$ and
we proved the claim.

By Theorem \ref{zzthm3.12}
$$\begin{aligned}
H_{\knx^{\mathfrak{S}_n}}(t)&= H_{\bar{B}}(t)
=\frac{(1-t^2)(1-t^6)(1-t^{10})\cdots (1-t^{4n-2})}
 {(1-t)(1-t^2)(1-t^3) \cdots (1-t^{2n-1})(1-t^{2n})}\\
&=\frac{\prod_{s=\lfloor \frac{n-1}{2}\rfloor+2}^n (1-t^{4s-2})}
{\prod_{j=1}^{\lfloor \frac{n}{2}\rfloor}(1-t^{4j}) \prod_{i=1}^n
(1-t^{2i-1})}
\end{aligned}
$$
which is an expression satisfying the condition in Definition
\ref{zzdef3.13}(2). Hence
$$cyc(\knx^{\mathfrak{S}_n})=\lfloor \frac{n}{2}\rfloor.$$
The assertion follows from the claim and the fact
$cci^+(A)\geq cyc(A)$.
\end{proof}

\section{Invariants under $\mathfrak{A}_n$}
\label{zzsec4}

First let us review the classical case. Let $\mathfrak{A}_n$ be
the alternating group. Any element of
$\kpx^{\mathfrak{A}_n}$ can be written uniquely as $h_1 + D h_2$,
where $h_1$ and $h_2$ are
symmetric polynomials and $D$ is the ``Vandermonde determinant"
$$D= D(x_1, \cdots, x_n)= \prod_{i < j}(x_i-x_j)$$
\cite[p. 5]{S1}.  Hence the maximal degree of a minimal set of
generators of  $\kpx^{\mathfrak{A}_n}$ is
$\dbinom{n}{2}.$
A polynomial $f$ is called ``antisymmetric" if $\tau f = -f$ for
every odd permutation $\tau \in \mathfrak{S}_n$ \cite[p. 5]{S1}; $D$
is the smallest degree antisymmetric element of
$\kpx^{\mathfrak{A}_n}$. Moreover, $D^2$ is a symmetric polynomial,
the Hilbert series of $\kpx^{\mathfrak{A}_n}$ is
$$\frac{1+t^{r}}{\prod_{i=1}^n(1-t^i)} =
\frac{1-t^{2r}}{(1-t^r)\prod_{i=1}^n(1-t^i)}$$
for $r= \dbinom{n}{2}$ \cite[pp. 104-5]{Be}, and hence
$\kpx^{\mathfrak{A}_n}$ is isomorphic to the complete intersection
$$\frac{k[\sigma_1, \ldots, \sigma_n][y]}{(y^2-D^2)}$$
under the map that associates $y$ to $D$ (and the symmetric polynomial
in the $x_i$ to $\sigma_i$). Following Definition \ref{zzdef3.13},
one easily gets
$$cci(\kpx^{\mathfrak{A}_n})=cci^+(\kpx^{\mathfrak{A}_n})=
cyc(\kpx^{\mathfrak{A}_n})=1.$$
The group $\mathfrak{A}_n$ is generated by 3-cycles, which have trace
$${\rm Tr}_{\kpx}(g,t)= \frac{1}{(1-t^3)(1-t)^{n-3}},$$
and hence are bireflections of $\kpx$; the 3-cycles are a generating
set of bireflections that the Kac-Watanabe-Gordeev Theorem states
must exist since $\kpx^{\mathfrak{A}_n}$ is a complete intersection.

In this section we consider the analogous situation for
$\knx^{\mathfrak{A}_n}$ for $n\geq 3$. As a general setup, we are
working with the noncommutative algebra $\knx$ unless otherwise
stated. Again there is an overlap between \cite{CA} and this
section.

The trace of a 3-cycle $g$ acting on $\knx$
is also $${\rm Tr}_{\knx}(g,t) =\frac{1}{(1-t^3)(1-t)^{n-3}},$$
hence $\mathfrak{A}_n$ is generated by quasi-bireflections of
$\knx$. The aim of this section is to show that
$\knx^{\mathfrak{A}_n}$ is a cci, which
is consistent with the conjectured generalization of the
Kac-Watanabe-Gordeev Theorem. Here the smallest degree antisymmetric
polynomial is $\mathcal{O}_{\mathfrak{A}_n} (x_1x_2 \cdots
x_{n-1})$, and the subring of invariants $\knx^{\mathfrak{A}_n}$ is
generated by $\mathcal{O}_{\mathfrak{A}_n} (x_1x_2 \cdots x_{n-1})$,
and either the $n-1$ super-symmetric polynomials $S_1, \ldots,
S_{n-1}$ or the power sums $P_1, \ldots, P_{2n-3}$, and so an upper
bound on the degrees of generators of $\knx^{\mathfrak{A}_n}$ is
$2n-3$. We will show that the Hilbert series of
$\knx^{\mathfrak{A}_n}$ is given by
\begin{equation}  \label{E4.0.1}\tag{E4.0.1}
H_{\knx^{\mathfrak{A}_n}}(t)\quad =\quad
\frac{(1+t) (1+t^{3}) \cdots (1+t^{2n-3})
(1+t^n)(1+t^{n-1})}{(1-t^2)(1-t^4)\cdots(1-t^{2n})}.
\end{equation}

We construct invariants under $\mathfrak{A}_n$ as
$\mathcal{O}_{\mathfrak{A}_n}(X^I)$, the sum of the orbit of a
monomial $X^I$ under $\mathfrak{A}_n$ [Definition \ref{zzdef3.2}];
we note that the number of terms in this sum is the index of the
$\mathfrak{A}_n$-stabilizer of $X^I$ in $\mathfrak{A}_n$.

\begin{lemma}\cite[Lemma 4.1.1]{CA}
\label{zzlem4.1}
If there is an odd permutation that stabilizes
$X^I$ then $f = \mathcal{O}_{\mathfrak{A}_n}(X^I)$ is also invariant
under the full symmetric group $\mathfrak{S}_n$.
\end{lemma}

\begin{proof}
Since the index of the subgroup $stab_{\mathfrak{A}_n}(X^I)$ in
$stab_{\mathfrak{S}_n}(X^I)$ is less than or equal to
$[\mathfrak{S}_n : \mathfrak{A}_n]=2$, if there is an odd
permutation that stabilizes $X^I$ then the index
$[stab_{\mathfrak{S}_n}(X^I): stab_{\mathfrak{A}_n}(X^I)]= 2$, and
the order of the orbit $X^I$ under $\mathfrak{S}_n =
[\mathfrak{S}_n: stab_{\mathfrak{S}_n}(X^I)]$ is the same as the
order of the orbit of $X^I$ under $\mathfrak{A}_n = [\mathfrak{A}_n:
stab_{\mathfrak{A}_n}(X^I)]$, so the orbit sum of $X^I$ under
$\mathfrak{S}_n$ is the same as that under $\mathfrak{A}_n$; hence
$\mathcal{O}_{\mathfrak{A}_n}(X^I)$, the orbit sum of $X^I$ under
$\mathfrak{A}_n$, is $\mathfrak{S}_n$-invariant.
\end{proof}

Here is an immediate consequence.

\begin{corollary}
\label{zzcor4.2} If $I=(i_j)$ with at least 2 indices $i_j = i_k$,
an even number, then $\mathcal{O}_{\mathfrak{A}_n}(X^I)$ is an
$\mathfrak{S}_n$-invariant.  In particular if there are at least 2
indices $i_j = i_k = 0$ then  $\mathcal{O}_{\mathfrak{A}_n}(X^I)$ is
an $\mathfrak{S}_n$-invariant.
\end{corollary}

\begin{lemma}\cite[Lemma 4.1.2]{CA}
\label{zzlem4.3} An $\mathfrak{A}_n$-orbit sum
$\mathcal{O}_{\mathfrak{A}_n}(X^I) = 0$ if and only if $I$ has at
least two indices $i_j=i_k$ an even number, and two indices $i_r =
i_s$ an odd number.
\end{lemma}

\begin{proof} If $I$ has repeated even indices then
$\mathcal{O}_{\mathfrak{A}_n}(X^I) =
\mathcal{O}_{\mathfrak{S}_n}(X^I)$ by Corollary \ref{zzcor4.2}.
Since $I$ has repeated odd indices then
$\mathcal{O}_{\mathfrak{S}_n}(X^I) =0$ by Lemma \ref{zzlem3.5}.

Conversely, suppose that $\mathcal{O}_{\mathfrak{A}_n}(X^I) = 0$,
then $X^I$ and $-X^I$ are in the $\mathfrak{A}_n$-orbit of $X^I$,
hence in the $\mathfrak{S}_n$-orbit of $X^I$. Hence for every $\tau
X^I$ in the $\mathfrak{S}_n$-orbit of $X^I$ we also have $-\tau X^I$
in the $\mathfrak{S}_n$-orbit, and so the $\mathfrak{S}_n$ orbit sum
is $0$, which forces at least two indices to have the same odd value
by Lemma \ref{zzlem3.5}. We have $\tau X^I = -X^I$ for an even
permutation $\tau$.  Write $\tau$ as a product of disjoint cycles
$$\tau = \nu_1 \cdots \nu_{2m} \mu_1 \cdots \mu_k$$
where the $\nu_i$ are odd permutations and the $\mu_j$ are even
permutations.  Note that since $\tau X^I = -X^I$ exponents in $I$
must be constant over the support of each cycle.  Suppose there are
no repeated even indices in $I$, so that all repeats are of odd
indices.  Hence for each $\mu_j = (a_1, \cdots, a_{2s_j+1})$, an
even cycle, $\mu_j$ can be written as an even number of
transpositions, interchanging variables with the same odd exponent.
By the proof of Lemma \ref{zzlem3.5} each of these transpositions
maps $X^I$ to $-X^I$, and hence $\mu_j X^I = X^I$, For similar
reasons each $\nu_i X^I =-X^I$.  It follows that $\tau X^I = \nu_1
\cdots \nu_{2m} \mu_1 \cdots \mu_k X^I = X^I$, a contradiction.
Hence $I$ must also contain at two indices with the same even
number.
\end{proof}

Note that $\mathfrak{A}_n$-orbit sums do not necessarily correspond
to partitions, e.g. when $n=4$ the orbit sums
$\mathcal{O}_{\mathfrak{A}_n}(x_1^4x_2^3x_3^2x_4)$ and
$\mathcal{O}_{\mathfrak{A}_n}(x_1^4x_2^3x_3x_4^2)$ are different
(and \linebreak[4] $\mathcal{O}_{\mathfrak{A}_n}(x_1^4x_2^3x_3^2x_4)
+ \mathcal{O}_{\mathfrak{A}_n}(x_1^4x_2^3x_3 x_4^2) =
\mathcal{O}_{\mathfrak{S}_n}(x_1^4x_2^3x_3^2x_4)$).

Adapting the classical definition, an element $g\in \knx$ is called
{\it symmetric} (respectively, {\it antisymmetric}) if $\tau (g)=g$
(respectively, $\tau (g)=-g$) for every odd permutation $\tau\in
\mathfrak{S}_n$. Note that $g$ is symmetric if and only if $g$ is
$\mathfrak{S}_n$-invariant. If $g$ is antisymmetric, then $g$ is
$\mathfrak{A}_n$-invariant. The following lemma follows easily.

\begin{lemma}
\label{zzlem4.4} Let $f,g,h$ be elements in $\knx$.
\begin{enumerate}
\item
Linear combinations of antisymmetric invariants are antisymmetric.
Hence if $f+g$ and $g$ are antisymmetric invariants, then $f$ is an
antisymmetric invariant.
\item
If $f = gh$ with $g$ an antisymmetric invariant and $h$ a symmetric
invariant, then $f$ is an antisymmetric invariant.
\item
If $f = gh$ with $f$ and $g$ antisymmetric invariants
then $h$ a symmetric invariant.
\end{enumerate}
\end{lemma}

The following lemma follows as in the case of $\kpx$ and the proof
is omitted.

\begin{lemma}\cite[Theorem 4.1.4]{CA}
\label{zzlem4.5} If $f$ is an $\mathfrak{A}_n$-invariant and
$\sigma$ is the transposition $(1,2)$ then $\sigma f = \tau f$ for
any odd permutation $\tau$. Furthermore $f + \sigma f$ is symmetric
and $f-\sigma f$ is antisymmetric. As a consequence, each invariant
$f \in \knx^{\mathfrak{A}_n}$ can be be written uniquely as the sum
of a symmetric invariant and an antisymmetric invariant.
\end{lemma}


\begin{example}\label{zzex4.6}
For $n \geq 3$ the following are examples of antisymmetric
invariants:  some ${\mathfrak{A}_n}$-orbits are antisymmetric (e.g.
$\mathcal{O}_{\mathfrak{A}_n}(x_1^4 x_2 x_3 )$ and
$\mathcal{O}_{\mathfrak{A}_n}(x_1^3 x_2^3 x_3 )$) and antisymmetric
elements can be constructed from the lemma above (e.g.  $f-\sigma f$
for $f=\mathcal{O}_{\mathfrak{A}_n}(x_1^4 x_2^3 x_3^2$ )).
\end{example}

For the rest of this section, we assume that $n\geq 3$ as
$\mathfrak{A}_2$ is trivial. In the case $\knx$ we have the two
antisymmetric orbit sums given in the lemma below; the orbit sums of
these monomials are symmetric polynomials when $\mathfrak{A}_n$ acts
on $\kpx$.

\begin{lemma}
\label{zzlem4.7}
The $\mathfrak{A}_n$ orbit sums
$$\mathcal{O}_{\mathfrak{A}_n}(x_1
x_2 \cdots x_n) = x_1 x_2 \cdots x_n\quad {\text{ and }}\quad
\mathcal{O}_{\mathfrak{A}_n}(x_1 x_2 \cdots x_{n-1} )$$ are both
antisymmetric $\mathfrak{A}_n$-invariants. And
$\mathcal{O}_{\mathfrak{A}_n}(x_1 x_2 \cdots x_{n-1} )$ is the
smallest degree antisymmetric invariant.
\end{lemma}

\begin{proof} It is easy to show that $x_1 x_2 \cdots x_n$ is an
antisymmetric $\mathfrak{A}_n$-invariant, whence
$\mathcal{O}_{\mathfrak{A}_n}(x_1 x_2 \cdots x_n) = x_1 x_2 \cdots
x_n$. So we focus on $\mathcal{O}_{\mathfrak{A}_n}(x_1 x_2 \cdots
x_{n-1} )$.


We note that
$$\mathcal{O}_{\mathfrak{A}_3}(x_1x_2) = x_1 x_2 - x_1x_3 + x_2 x_3.$$
For $n \geq 4$  applying the even permutation $(1,2) (n-1,n)$ to
$x_1x_2 \cdots x_{n-1}$ we obtain
$$(1,2) (n-1,n) (x_1x_2 \cdots x_{n-1}) =
x_2x_1 \cdots x_{n-2}x_n = - x_1x_2 \cdots x_{n-2}x_n$$ and similarly
$$(1,2) (n-2,n) (x_1x_2 \cdots x_{n-1}) =
x_2x_1 \cdots x_{n-3} x_{n}x_{n-1} =  x_1x_2 \cdots x_{n-3}x_{n-1}x_n$$
and
\begin{eqnarray*}
(1,2) (j,n) (x_1x_2 \cdots x_{n-1}) &=& x_2x_1
\cdots x_{j-1} x_{n}x_{j+1} \cdots x_n \\
& =& (-1)^{n-j} x_1x_2 \cdots x_{j-1}x_{j+1} \cdots x_n
\end{eqnarray*}
so that the $n$ monomials with $j$th missing variable  occur in the
$\mathfrak{A}_n$-orbit with the sign $(-1)^{n-j}$. Since $x_1 \cdots
x_{n-1}$ has repeated odd exponents we have seen that the monomials
in the $\mathfrak{S}_n$-orbit of $x_1 \cdots x_{n-1}$ occur with
both plus and minus signs, and the $\mathfrak{S}_n$-orbit sum of
$x_1x_2 \cdots x_{n-1}$ is $0$.  Hence the $\mathfrak{S}_n$-orbit of
$x_1 \cdots x_{n-1}$ has $2n$ elements, and so the
$\mathfrak{S}_n$-stabilizer of $x_1 \cdots x_{n-1}$ has $(n-1)!/2$
elements, and clearly the $(n-1)!/2$ even permutations of $\{1,
\ldots, n-1\}$ stabilize $x_1 \cdots x_{n-1}$ so must constitute its
stabilizer.  Hence the stabilizer in $\mathfrak{A}_n$ must also have
$(n-1)!/2$ elements, and hence the $\mathfrak{A}_n$-orbit of $x_1
\cdots x_{n-1}$ must be the $n$ elements we have computed, and hence
$$\mathcal{O}_{\mathfrak{A}_n}(x_1 \cdots x_{n-1}) = (x_1 \cdots
x_{n-1}) - (x_1 \cdots x_{n-2} x_n) + (x_1 \cdots x_{n-3} x_{n-1} x_n)$$
 $$ + \cdots + ((-1)^{n-1}x_2x_3 \cdots x_n). $$
Then to see the effect of any transposition $(i,j)$ on this orbit
sum, consider a summand of the orbit sum that contains both $i$ and
$j$ and note, as in the argument above, that the transposition
$(i,j)$ changes the sign of this term; since any element in an orbit
represents the orbit, any transposition reverses the sign on the
$\mathfrak{A}_n$-orbit sum of $x_1 \cdots x_{n-1}$, and hence
$\mathcal{O}_{\mathfrak{A}_n}(x_1 x_2 \cdots x_{n-1} )$ is an
antisymmetric $\mathfrak{A}_n$-invariant.

There can be no smaller degree antisymmetric
$\mathfrak{A}_n$-invariant since any smaller degree monomial $X^I$
must have at least two zero entries in $I$, hence $\mathcal{O}(X^I)$
must be $\mathfrak{S}_n$-symmetric, and so no linear combination of
such orbits can be antisymmetric.
\end{proof}

The antisymmetric orbit sum $\mathcal{O}_{\mathfrak{A}_n}(x_1 \cdots
x_n)$ can be generated from the super-symmetric polynomials and
$\mathcal{O}_{\mathfrak{A}_n}(x_1 \cdots x_{n-1})$.

\begin{lemma} \label{zzlem4.8}
The antisymmetric orbit sum $\mathcal{O}_{\mathfrak{A}_n} (x_1
\cdots x_n)=x_1 \cdots x_n$ is generated by the super-symmetric
polynomial $P_1=S_1=\mathcal{O}_{\mathfrak{S}_n}
(x_1)=\mathcal{O}_{\mathfrak{A}_n} (x_1)$ and the antisymmetric
orbit sum $\mathcal{O}_{\mathfrak{A}_n} (x_1 \cdots x_{n-1})$ as
follows
$$\mathcal{O}_{\mathfrak{A}_n}(x_1 \cdots x_n) =\frac{1}{2n}
(\mathcal{O}_{\mathfrak{A}_n}(x_1 \cdots x_{n-1}) S_1 + (-1)^{n-1}
S_1\mathcal{O}_{\mathfrak{A}_n} (x_1 \cdots x_{n-1}) ).$$
\end{lemma}

\begin{proof} Computing
$$\begin{aligned}
\mathcal{O}_{\mathfrak{A}_n}&(x_1 \cdots x_{n-1})
\mathcal{O}_{\mathfrak{A}_n}(x_1)\\
&= (x_1x_2 \cdots x_{n-1} -  x_1x_2 \cdots x_{n-2}x_n + \cdots +
(-1)^{n-1}x_2 \cdots x_n) (x_1 + \cdots + x_n)\end{aligned}$$ we see
that the monomial $x_1 \cdots x_n$ occurs $n$ times (each with
positive sign) as a summand in this product when expanded, and
$$x_1 \cdots x_{n-1}x_1 = (-1)^{n-2}x_1^2 \cdots x_{n-1}$$
so the respective orbits sums occur in the expanded product.
Since there are $n^2$ monomials in the product
$\mathcal{O}_{\mathfrak{A}_n}(x_1 \cdots x_{n-1})
\mathcal{O}_{\mathfrak{A}_n}(x_1)  $,
and $n(n-1)$ summands in $\mathcal{O}_{\mathfrak{A}_n}
(x_1^2x_2 \cdots x_{n-1})$ these orbit sums account for
all the terms, and so
$$\mathcal{O}_{\mathfrak{A}_n}(x_1 \cdots x_{n-1})
\mathcal{O}_{\mathfrak{A}_n}(x_1) = n\mathcal{O}_{\mathfrak{A}_n}
(x_1 \cdots x_n) + (-1)^{n-2}\mathcal{O}_{\mathfrak{A}_n}
(x_1^2x_2 \cdots x_{n-1}).$$
Similarly
$$ \mathcal{O}_{\mathfrak{A}_n}(x_1)\mathcal{O}_{\mathfrak{A}_n}
(x_1 \cdots x_{n-1}) =\mathcal{O}_{\mathfrak{A}_n}(x_1^2x_2
\cdots x_{n-1}) + (-1)^{n-1}n\mathcal{O}_{\mathfrak{A}_n}
(x_1 \cdots x_n),  $$
and the result follows.
\end{proof}

Next we note that the super-symmetric polynomial
$S_{n} = \mathcal{O}_{\mathfrak{S}_n}(x_1^2 \cdots x_{n-1}^2x_n)$
can be generated by antisymmetric invariants
$\mathcal{O}_{\mathfrak{A}_n}(x_1 \cdots x_{n})$ and
$\mathcal{O}_{\mathfrak{A}_n}(x_1 \cdots x_{n-1})$.

\begin{lemma} \label{zzlem4.9}
The super-symmetric polynomial $S_n = \mathcal{O}_{\mathfrak{S}_n}
(x_1^2 \cdots x_{n-1}^2x_n)$ can be generated by antisymmetric
invariants $\mathcal{O}_{\mathfrak{A}_n}(x_1 \cdots x_{n})$ and
$\mathcal{O}_{\mathfrak{A}_n}(x_1 \cdots x_{n-1})$ as follows
$$\mathcal{O}_{\mathfrak{S}_n}
(x_1^2 \cdots x_{n-1}^2x_n)  = (-1)^{(n-2)(n-1)/2}
(\mathcal{O}_{\mathfrak{A}_n}(x_1 \cdots x_{n-1}))
(\mathcal{O}_{\mathfrak{A}_n}(x_1 \cdots x_{n})).$$
\end{lemma}

\begin{proof} The monomial $x_1^2 \cdots x_{n-1}^2x_n$ is
stabilized by $(1,2)$ so
\[ S_{n} = \mathcal{O}_{\mathfrak{S}_n}(x_1^2 \cdots x_{n-1}^2x_n)
= \mathcal{O}_{\mathfrak{A}_n}(x_1^2 \cdots x_{n-1}^2x_n), \]
 and
$$S_{n} = \sum_{i=1}^n x_1^2 \cdots x_{i-1}^2 x_i x_{i+1}^2 \cdots x_n^2.$$
This expression is a sum of $n$ terms, each with $x_1 \cdots x_n$ as a factor.
Consider the product $ (\mathcal{O}_{\mathfrak{A}_n}
(x_1 \cdots x_{n-1})) (x_1 \cdots x_{n})$, and observe
when this product is expanded one term is
\begin{eqnarray*} (x_1 \cdots x_{n-1})(x_1 \cdots x_{n})& =
& (-1)^{n-2}(x_1^2 x_2 \cdots x_{n-1})(x_2 \cdots x_n) \\
& = & (-1)^{n-2} (-1)^{n-3} (x_1^2 x_2^2 \cdots x_{n-1})(x_3 \cdots x_n) \\
& = &(-1)^{(n-2)(n-1)/2}(x_1^2 \cdots x_{n-1}^2x_n),
\end{eqnarray*}
the last equality holding by induction. Since
$(\mathcal{O}_{\mathfrak{A}_n}(x_1 \cdots x_{n-1})) (x_1 \cdots
x_{n})$ is an invariant, the entire orbit sum of this monomial must
occur as terms in this expanded product, accounting for the $n$
terms in $S_{n}$  yielding the result.
\end{proof}

Here we are ready to prove a result of Cameron Atkins \cite{CA}.

\begin{theorem}\cite[Theorem 4.2.7]{CA}
\label{zzthm4.10}
The  fixed subring $\knx^{\mathfrak{A}_n}$ is generated by the
super-symmetric polynomials $S_1, \cdots, S_{n-1}$ and the antisymmetric
$\mathfrak{A}_n$-invariant $\mathcal{O}_{\mathfrak{A}_n}(x_1 \cdots
x_{n-1})$ (or the odd power sums $P_1, P_3, \ldots , P_{2n-3}$ and
$\mathcal{O}_{\mathfrak{A}_n}(x_1 \cdots x_{n-1})$).
\end{theorem}

\begin{proof}
By Lemma~\ref{zzlem4.8} $S_1=P_1$ and $\mathcal{O}(x_1 \cdots
x_{n-1})$ generate $x_1x_2 \cdots x_n$, which in turn by
Lemma~\ref{zzlem4.9} generate $S_n$.  By Theorem~\ref{zzthm3.10}
$S_1, S_2, \ldots , S_n$ generate all the symmetric invariants.
Hence it suffices to show that any antisymmetric
$\mathfrak{A}_n$-invariant $f$ can be obtained.

We will induct on the degree of $f$, noting that the result is true
in degrees $\leq n-1$ since $\mathcal{O}(x_1 \cdots x_{n-1})$ is the
only antisymmetric $\mathfrak{A}_n$-invariant of degree $\leq n-1$.

Let $X^I$ be the leading term of $f$ under the length-lexicographic
order.  If $\sigma$ is an transposition $\sigma f = -f$ also has
leading term $X^I$.  Hence by applying transpositions, we may assume
that $f$ has leading term $X^I$ where $I$ is weakly decreasing (and
hence corresponds to a partition). We can write $f$ as a linear
combination of distinct orbit sums $f= \sum c_I
\mathcal{O}_{\mathfrak{A}_n}(X^I)$ where $X^I$ is the highest degree
monomial in the orbit, and where $c_I\in k$ [Lemma \ref{zzlem3.4}].
Since $f$ is antisymmetric $\sigma f = \sum c_I \sigma
\mathcal{O}_{\mathfrak{A}_n}(X^I) = -f$ so that $2f = f -(-f) = \sum
c_I (\mathcal{O}_{\mathfrak{A}_n}(X^I) - \sigma
\mathcal{O}_{\mathfrak{A}_n}(X^I))$. Hence without loss of
generality we may assume that $f = \mathcal{O}_{\mathfrak{A}_n}(X^I)
- \sigma(\mathcal{O}_{\mathfrak{A}_n}(X^I))$, with $X^I$ the leading
term of $f$, and with $I = (\lambda_i)$ weakly decreasing; (since
$X^I$ is the leading term of $f$, and $f$ is antisymmetric,
$\sigma(\mathcal{O}_{\mathfrak{A}_n}(X^I)) \neq
\mathcal{O}_{\mathfrak{A}_n}(X^I)$).    Since $x_1^2 \cdots x_n^2$
is central and symmetric, we can factor it out of $f$, obtaining an
antisymmetric invariant of smaller degree. Hence we may assume
without loss of generality that $\lambda_n = 0$ or $1$.  If
$\lambda_n=1$, then each $x_i$ occurs in all terms of $f$, so we can
factor out $(x_1 \cdots x_n)$ from $f$ and write $f= h (x_1 \cdots
x_n)$ for some $\mathfrak{A}_n$-invariant $h$.  It follows that $h$
is symmetric, and we are done.  Hence, assume that $\lambda_n = 0$
and $I = (\lambda_1, \ldots, \lambda_{n-1}, 0)$.

Now we induct on the order of $I$. The lowest order possible for $I$
is when $I=(\lambda_1,0,\cdots,0)$. Since $n\geq 3$, we have
$\lambda_{n-1}=0$. If $\lambda_{n-1} = 0$, then the transposition
$\tau = (n-1,n)$ stabilizes $X^I$ and hence $\mathcal{O}(X^I)$ is
$\mathfrak{S}_n$-invariant [Lemma \ref{zzlem4.1}]. Consequently,
$f=0$ and we are done. Therefore we can assume that $\lambda_i \neq
0$ for all $i=1,\cdots,n-1$.
 Let $I^* = (\lambda_1-1,
\lambda_2-1 , \ldots, \lambda_{n-1}-1,0)$, which is a weakly
decreasing sequence, and let $h =
\mathcal{O}_{\mathfrak{A}_n}(X^{I^*}) + \sigma
\mathcal{O}_{\mathfrak{A}_n}(X^{I^*})$, which is
$\mathfrak{S}_n$-invariant (it is possible that $ \mathcal{O}(X^I)$
itself is $\mathfrak{S}_n$-invariant -- e.g. if $\lambda_{n-1} = 1$
or $I^*$ has two even entries that are equal -- in this case $h = 2
\mathcal{O}_{\mathfrak{A}_n}(X^{I^*})$).  Let $g= h
\mathcal{O}_{\mathfrak{A}_n}(x_1 \cdots x_{n-1})$, which is an
antisymmetric $\mathfrak{A}_n$-invariant that is a product of a
$\mathfrak{S}_n$-invariant and $\mathcal{O}_{\mathfrak{A}_n}(x_1
\cdots x_{n-1})$. We claim that $\pm f$ is a summand of $g$ and that
all other terms have lower order; by induction these claims will
complete the proof. Notice that the terms $g_1$ and $g_2$ occur in
$g$ where
$$g_1 = (x_1^{\lambda_1 - 1} x_2^{\lambda_2 - 1} \cdots
x_{n-1}^{\lambda_{n-1} - 1}) (x_1 \cdots x_{n-1}) $$
$$g_2 = (x_2^{\lambda_1 - 1} x_1^{\lambda_2 - 1} \cdots
x_{n-1}^{\lambda_{n-1} - 1}) (x_1 \cdots x_{n-1}),$$ and hence their
$\mathfrak{A}_n$-orbit sums occur in $g$. Note that $g_1 = \pm X^I$
and $g_2 =\pm \sigma X^I$  and $\sigma g_1 = - g_2$, and hence $\pm
f$ is a summand of $g$.  Finally notice that $X^I$ is clearly the
leading term of $g$ and so all the other terms of $g$ are of lower
order.  Hence $f \pm g$ is antisymmetric of lower order, hence of
the desired form by induction.

The argument of Lemma~\ref{zzlem3.9} shows that $S_1, S_2, \ldots ,
S_{n-1}$ can be obtained from $P_1, P_3, \ldots , P_{2n-3}$.
\end{proof}

In the above proof we have shown that antisymmetric invariants
correspond to partitions
$$I \mapsto \mathcal{O}_{\mathfrak{A}_n}(X^I) -
\sigma \mathcal{O}_{\mathfrak{A}_n}(X^I)$$ for any odd permutation
$\sigma$.  This antisymmetric invariant will be non-zero if and only
if $0 \neq \mathcal{O}_{\mathfrak{A}_n}(X^I)$ is not
$\mathfrak{S}_n$-invariant, i.e. $\mathcal{O}_{\mathfrak{A}_n}(X^I)$
has no odd permutations stabilizing it.  By the lemma below this is
equivalent to $I$ having no repeated even indices (by Lemma
\ref{zzlem4.3} this condition also assures
$\mathcal{O}_{\mathfrak{A}_n}(X^I) \neq 0$.)

\begin{lemma}
\label{zzlem4.11} Let $X^I$ be the highest degree lexicographic
ordered term in the $\mathfrak{A}_n$-orbit of $X^I$.  Then $\sigma
\mathcal{O}_{\mathfrak{A}_n}(X^I) =
\mathcal{O}_{\mathfrak{A}_n}(X^I)$ for an odd permutation $\sigma$
if and only if $I$ has at least two entries
 $\lambda_j = \lambda_k$ that are an even number (including $0$).
\end{lemma}

\begin{proof}
If $\lambda_j= \lambda_k$ is even then $(j,k) X^I = X^I$ so
$\mathfrak{S}_n = \mathfrak{A}_n \cup \mathfrak{A}_n (j,k)$ and the
$\mathfrak{A}_n$-orbit of $X^I$ is the same as the
$\mathfrak{S}_n$-orbit of $X^I$ so $(j,k)
\mathcal{O}_{\mathfrak{A}_n}(X^I)
=\mathcal{O}_{\mathfrak{A}_n}(X^I)$, and, in fact, any permutation
stabilizes the orbit sum.

Conversely, suppose that there is an odd permutation $\sigma$ with
$\sigma \mathcal{O}_{\mathfrak{A}_n}(X^I) =
\mathcal{O}_{\mathfrak{A}_n}(X^I)$. Since $\sigma X^I$ is in the
$\mathfrak{A}_n$-orbit of $X^I$ we must have $\sigma X^I = \tau X^I$
for $\tau$ an even permutation. Hence $\tau^{-1} \sigma X^I = X^I$
so $X^I$ is stabilized by an odd permutation.  Suppose that $I$ has
no repeated even entries, and  write $\sigma = \nu_1 \cdots
\nu_{2m+1} \mu_1 \cdots \mu_k$ as a product of disjoint cycles,
where $\nu_i$ are odd permutations and $\mu_i$ are even.  Noting
that entries of $I$ in the support of each cycle must be constant
and all repeated entries are assumed to be odd, we see that each
$\mu_i X^I = X^I$ because $\mu_i$ is the product of an even number
of transpositions of variables with the same odd exponents and so
each transposition changes the sign; since there are an even number
of sign changes $\mu_i X^I = X^I$.  However $\nu_i X^I = -X^I$ since
$\nu_i$ is the product of an odd number of interchanges of variables
to the same odd power, and hence results in an odd number of sign
changes.  Hence
$$\sigma X^I = \nu_1 \cdots \nu_{2m+1} \mu_1 \cdots \mu_k X^I
= \nu_1 \cdots \nu_{2m+1} X^I = (-1)^{2m+1} X^I = -X^I,$$
contradicting $\sigma X^I = X^I$. Hence $I$ must have at least one
repeated even entry.
\end{proof}

We note that in the commutative case the antisymmetric nonzero
invariants $\mathcal{O}_{\mathfrak{A}_n}(X^I) - \sigma
\mathcal{O}_{\mathfrak{A}_n}(X^I)$ that corresponding to a partition
$I$ are those with all entries of $I$ distinct.

We next compute the Hilbert series for $\knx^{\mathfrak{A}_n}$ and
use it to show that $\knx^{\mathfrak{A}_n}$ is a cci.
For specific values of $n$ the coefficients of these
series do not seem to be in the {\it Online Encyclopedia of Integer
Sequences}.

\begin{lemma}
\label{zzlem4.12} The Hilbert series of $\knx^{\mathfrak{A}_n}$ is
given by
\begin{eqnarray*}
 H_{\knx^{\mathfrak{A}_n}}(t)&=& \frac{(1-t^2)(1-t^6)(1-t^{10})
 \cdots (1-t^{4n-2})(1+t^n)(1+t^{n-1})}{(1-t)(1-t^2)(1-t^3)
 \cdots (1-t^{2n-1})(1-t^{2n})(1+t^{2n-1})}.
 \end{eqnarray*}
\end{lemma}

\begin{proof}
By remarks above in each dimension the invariants are vector space
direct sums of the symmetric invariants and the antisymmetric
invariants, so the Hilbert series $H_{\knx^{\mathfrak{A}_n}}(t)$ for
the invariants under $\mathfrak{A}_n$ is the sum of
$H_{\knx^{\mathfrak{S}_n}}(t)$ and the generating function $S_n(t)$
for $s_n(k)$, the number of partitions of $k$ with at most $n$ parts
having no repeated even parts (not even $0$).  By
Proposition~\ref{xxpro6.3} of the Appendix we have
\[ S_n(t) = D_n(t) \frac{t^{n-1}(1+t)}{(1+t^{2n-1})}. \]
Hence
\begin{eqnarray*}
H_{\knx^{\mathfrak{A}_n}}(t) &=& D_n(t) + S_n(t) = D_n(t) + D_n(t)
\frac{t^{n-1}(1+t)}{(1+t^{2n-1})} \\
&=& D_n(t) \left(1 +  \frac{t^{n-1}(1+t)}{(1+t^{2n-1})}\right) \\
&=&D_n(t) \frac{(1+t^n)(1+t^{n-1})}{(1+t^{2n-1})} \\
&=& \frac{(1-t^2)(1-t^6)(1-t^{10})\cdots (1-t^{4n-2})(1+t^n)(1+t^{n-1})}
{(1-t)(1-t^2)(1-t^3) \cdots (1-t^{2n-1})(1-t^{2n})(1+t^{2n-1})}.
\end{eqnarray*}
Canceling yields the expression in equation~\eqref{E4.0.1}.
\end{proof}

Consider the algebras given by
\begin{eqnarray*}
B_{n-1}&=& k[p_1, \cdots, p_n][y_1: \tau_1, \delta_1] \cdots
[y_{n-1}: \tau_{n-1}, \delta_{n-1}]\\
B_{n+1}&=& B_{n-1}[y_{n+1};\tau_{n+1}]\\
B_{n+2}&=& B_{n+1}[y_{n+2};\tau_{n+2},\delta_{n+2}].
\end{eqnarray*}
For $i \leq n-1$ define $\tau_i$ and $\delta_i$ as for the algebra
$B$ considered in the previous section (note that $B$ is not a
subalgebra of $C$ since $y_n$ is not adjoined). Define the
$\tau_{n+1}$ by letting it be the identity on $R=k[p_1,
\cdots, p_n]$ and $\tau_{n+1}(y_i) = (-1)^{n-1}y_i$ for $i\leq n-1$.
Then $\tau_{n+1}$ extends uniquely to an algebra automorphism
of $B_{n-1}$. Define the algebra automorphism $\tau_{n+2}$ of $D_n$ by
letting it be the identity on $R$ and letting
$$\tau_{n+2}(y_i) = \begin{cases} (-1)^ny_i & {\text{if}} \; i\leq n-1,\\
(-1)^{n+1} y_{n+1} & {\text{if}} \; i=n+1.\end{cases}$$
The derivation $\delta_{n+2}$ is given by letting $\delta_{n+2}(a)=0$
for all $a \in R$, $\delta_{n+2}(y_i) = (-1)^{n-1} 2 n a_{2i-2}y_{n+1}$
for $i\leq n-1$, and $\delta_{n+2}(y_{n+1})=0.$  Recall that $a_{2i-2} =
f_{2i-2}(p_1, p_2, \ldots p_n)$ where $f_{2i-2}$ is given by
\eqref{E3.7.1}.

\begin{lemma}\label{zzlem4.13}
Retain the above notation.
\begin{enumerate}
\item
$\tau_{n+2}$ is an algebra automorphism of $B_{n+1}$.
\item
$\delta_{n+2}$ is a $\tau_{n+2}$-derivation of $B_{n+1}$.
\end{enumerate}
\end{lemma}

\begin{proof} (1) It is straightforward to check that
$\tau_{n+2}$ is an algebra automorphism of $B_{n+1}$.

(2) The relations of $B_{n+1}$ are of the form
$$\begin{aligned}
y_ia-a y_i &=0, \; \forall \; i=1,\cdots,n-1, n+1,\; a\in R\\
y_iy_j+y_jy_i& =2 \; a_{2i+2j-2}, \; \forall \; 1\leq i, j\leq n-1,\\
y_{n+1}y_i+ (-1)^n y_i y_{n+1}&=0, \; \forall i=1,\cdots, n-1.
\end{aligned}
$$
The proof of $\delta_{n+2}$ preserving the relations $y_ia-a y_i =0$
is similar to the proof of Lemma \ref{zzlem2.3}(2). Now we show that
$\delta_{n+2}$ preserves other relations. For $i,j\leq n-1$,
$$\begin{aligned}
\delta_{n+2}&(y_iy_j+y_jy_i-2a_{2i+2j-2})\\
&=
\delta_{n+2}(y_i)y_j+\tau_{n+2}(y_i)\delta_{n+2}(y_j)+
\delta_{n+2}(y_j)y_i+\tau_{n+2}(y_j)\delta_{n+2}(y_i) \quad \\
&= (-1)^{n-1}2n a_{2i-2} y_{n+1} y_j+(-1)^{n}y_i
(-1)^{n-1}2n a_{2j-2} y_{n+1}\\
&\quad +
(-1)^{n-1}2n a_{2j-2} y_{n+1} y_i+(-1)^{n}y_j
(-1)^{n-1}2n a_{2i-2} y_{n+1}\\
&=0
\end{aligned}
$$
For $i\leq n-1$, we have
$$\begin{aligned}
\delta_{n+2}&(y_{n+1}y_i+(-1)^{n}y_iy_{n+1})\\
&=\tau_{n+2}(y_{n+1})\delta_{n+2}(y_i)+
(-1)^n\delta_{n+2}(y_i)y_{n+1}\\
&= (-1)^{n+1}y_{n+1} (-1)^{n-1} 2n a_{2i-2} y_{n+1}
+(-1)^n (-1)^{n-1} 2n a_{2i-2} y_{n+1}y_{n+1}\\
&= 0.
\end{aligned}
$$
\end{proof}

The above lemma verifies that $\delta_{n+2}$ is a
$\tau_{n+2}$-derivation.
Let $C=B_{n+2}$. The algebra $C$ is AS regular of dimension $2n+1$.
Grade $C$ by letting degree$(y_i) = 2i-1$ for $i \leq n-1$,
degree$(y_{n+1})=n$, and degree$(y_{n+2})=n-1.$ Then the Hilbert
series of $C$ is given by
\[ H_C(t) =  \frac{1}{(1-t)(1-t^2)\cdots (1-t^{2n-3})(1-t^{2n-2})
(1-t^{2n})(1-t^n)(1-t^{n-1})}. \]
Since $\mathfrak{A}_n \leq \mathfrak{S}_n$, the algebra $k[x_1^2,
x_2^2, \ldots , x_n^2]^{\mathfrak{S}_n}$ is a subalgebra of
$\knx^{\mathfrak{A}_n}$. Then $k[x_1^2, x_2^2, \ldots ,
x_n^2]^{\mathfrak{S}_n} = k[\rho_1, \rho_2, \ldots, \rho_n],$ a
commutative polynomial ring where $\rho_i = \sigma_i(x_1^2, x_2^2,
\ldots , x_n^2)$ and $\sigma_i$ is the $i$th elementary polynomial.
Observe that
\begin{equation} \label{E3.12.1}\tag{E3.12.1}
\mathcal{O}_{\mathfrak{A}_n}(x_1 \cdots x_{n-1})^2 = \pm
\mathcal{O}_{\mathfrak{A}_n}((x_1 \cdots x_{n-1})^2) = \pm
\mathcal{O}_{\mathfrak{A}_n}(x_1^2 \cdots x_{n-1}^2)
\end{equation}
because
$$\begin{aligned}
&(x_1 \cdots x_k \cdots x_{n-1}) (x_1 \cdots x_{k-1} x_{k+1}\cdots x_{n})\\
& = (-1)^{n-2}(x_1 \cdots x_k \cdots x_{n-1} x_{n})
(x_1 \cdots x_{k-1} x_{k+1}\cdots x_{n-1})\\
& = (-1)^{n}(-1)^{(n-k) + (k-1)}(x_1 \cdots x_{k-1} x_{k+1}
\cdots x_{n-1} x_{n}) (x_1 \cdots x_{k-1} x_k
x_{k+1}\cdots x_{n-1})\\
& = - (x_1 \cdots x_{k-1} x_{k+1} \cdots x_{n-1} x_{n})
(x_1 \cdots x_{k-1} x_k x_{k+1}\cdots x_{n-1}),
\end{aligned}$$
so the orbits of
the cross-terms cancel out, leaving only an orbit in the $x_i^2$
that is symmetric. Hence we can write
$\mathcal{O}_{\mathfrak{A}_n}(x_1^2x_2^2\cdots x_{n-1}^2) =
g(\rho_1, \rho_2, \ldots , \rho_n)$ for a polynomial $g$. Similarly,
$(x_1x_2\cdots x_n)^2= \pm x_1^2x_2^2\cdots x_n^2 =h(\rho_1, \rho_2,
\ldots , \rho_n)$ for a polynomial $h$.

As in the previous section let for $i \leq n-1$ let $r_i = y_i^2 -
a_{4i-2}$. Let
\begin{equation}
\label{E4.13.1}\tag{E4.13.1}
b_1=g(p_1, p_2, \ldots , p_n)
\end{equation}
and
\begin{equation}
\label{E4.13.2}\tag{E4.13.2}
b_2=h(p_1, p_2, \ldots ,p_n)
\end{equation}
and consider two additional relations $r_{n+1}
=y_{n+1}^2-b_2$ and $r_{n+2} = y_{n+2}^2-b_1$.

The proof of the following lemma is the same as that of
Lemma~\ref{zzlem3.11}.

\begin{lemma}\label{zzlem4.14}
The sequence $\{r_1, r_2, \ldots , r_{n-1}, r_{n+1}, r_{n+2}\}$
is a central regular sequence in $C$.
\end{lemma}

We are now ready to show that $\knx^{\mathfrak{A}_n}$ is a cci.

\begin{theorem}
\label{zzthm4.15}
The algebra $\knx^{\mathfrak{A}_n}$ is a cci.
\end{theorem}
\begin{proof}
Note that $\mathcal{O}_{\mathfrak{A}_n}(x_1x_2\cdots x_{n-1})$ and
$x_1x_2\cdots x_n$ are elements of $\knx^{\mathfrak{A}_n}$. Consider
the algebra $C$ constructed above and define a map $\phi: C
\longrightarrow \knx^{\mathfrak{A}_n}$ as follows: for $i \leq n$
let $\phi(p_i) = \rho_i$; for $i\leq n-1$ let $\phi(y_i)=P_{2i-1}$;
let $\phi(y_{n+1}) =x_1x_2\cdots x_n$; and let $\phi(y_{n+2})=
\mathcal{O}_{\mathfrak{A}_n}(x_1x_2\cdots x_{n-1}).$  Note that
$\phi$ takes $k[p_1, p_2, \ldots , p_n]$ isomorphically onto
$k[\rho_1, \rho_2, \ldots , \rho_n]$. In the proof of
Theorem~\ref{zzthm3.12} it was shown that $\phi$ preserves the skew
polynomial relations associated to $y_i$ for $i \leq n-1$.
Calculating shows that $(x_1x_2\cdots x_n)P_{2i-1} =
(-1)^{n-1}(x_1x_2\cdots x_n)P_{2i-1}$, and hence $\phi$ preserves
the relation associated to $y_{n+1}$. Further calculation shows that
$$\begin{aligned}
\mathcal{O}_{\mathfrak{A}_n}(x_1x_2\cdots x_{n-1})P_{2i-1} &=
(-1)^nP_{2i-1}\mathcal{O}_{\mathfrak{A}_n}(x_1x_2\cdots x_{n-1}) \\
&\quad + (-1)^{n-1} 2 n P_{2i-2}\;\cdot (x_1x_2\cdots x_n).
\end{aligned}$$ Since $\mathcal{O}_{\mathfrak{A}_n}(x_1x_2\cdots
x_{n-1})(x_1x_2\cdots x_n) = (-1)^{n-1} (x_1x_2\cdots
x_n)\mathcal{O}_{\mathfrak{A}_n}(x_1x_2\cdots x_{n-1})$
and $P_{2i-2}=f_{2i-2}(\rho_1,\rho_2, \ldots , \rho_n)$, the
relation associated
to $y_{n+2}$ is preserved by $\phi$. Hence $\phi$ is a graded ring
homomorphism.  The homomorphism $\phi$ is onto by
Theorem~\ref{zzthm4.10}. By \eqref{E3.12.1}
\begin{eqnarray*}
 0 &= &\mathcal{O}_{\mathfrak{A}_n}(x_1x_2\cdots x_{n-1})^2 -
 \mathcal{O}_{\mathfrak{A}_n}(x_1^2x_2^2\cdots x_{n-1}^2) \\
 & =& \mathcal{O}_{\mathfrak{A}_n}(x_1x_2\cdots x_{n-1}) -
g(\rho_1, \rho_2. \ldots , \rho_n) = \phi(y_{n+2}^2-b_1) =\phi(r_{n+2}).
\end{eqnarray*}
Similarly, $\phi(r_{n+1}) = \phi(y_{n+1}^2-b_2)=0.$  As in the proof
of Theorem~\ref{zzthm3.12} $\phi(r_i) =0$ for $i \leq n-1$.  Hence
$(r_1, r_2, \ldots , r_{n-1}, r_{n+1}, r_{n+2})\subseteq \ker(\phi)$,
and $\phi$ induces a graded ring homomorphism
$\bar{\phi}: \overline{C} \longrightarrow \knx^{\mathfrak{A}_n}$ where
$\overline{C} = C/(r_1, r_2, \ldots , r_{n-1}, r_{n+1},r_{n+2}).$

We have degree$(r_i) = 4i-2$ for $i\leq n-1$, degree$(r_{n+1}) =
2n$, and degree$(r_{n+2}) =2n-2$. Since $\{r_1, r_2, \ldots
,r_{n-1}, r_{n+1}, r_{n+2}\}$ is a regular sequence, the Hilbert
series of $\overline{C}$ is given by
\begin{eqnarray*}
H_{\overline{C}}(t)&=& \frac{(1-t^2)(1-t^6)(1-t^{10})\cdots
(1-t^{4n-6})(1-t^{2n})(1-t^{2n-2})}{(1-t)(1-t^2)
\cdots (1-t^{2n-2})(1-t^{2n})(1-t^n)(1-t^{n-1})}\\
&=&\frac{(1-t^2)(1-t^6)(1-t^{10})\cdots (1-t^{4n-6})(1-t^{2n})
(1-t^{2n-2})}{(1-t)(1-t^2) \cdots (1-t^{2n-2})(1-t^{2n})(1-t^n)(1-t^{n-1})}
\frac{(1-t^{4n-2})}{(1-t^{4n-2})} \\
&=&\frac{(1-t^2)(1-t^6)(1-t^{10})\cdots (1-t^{4n-2})(1+t^n)(1+t^{n-1})}
{(1-t)(1-t^2)(1-t^3) \cdots (1-t^{2n-1})(1-t^{2n})(1+t^{2n-1})}.
\end{eqnarray*}
This is the Hilbert series of $\knx^{\mathfrak{A}_n}$, and hence the
ring homomorphism $\bar{\phi}$ is an isomorphism as desired. The
assertion follows.
\end{proof}

\begin{theorem}
\label{zzthm4.16} $\lfloor \frac{n}{2}\rfloor=cyc(\knx^{\mathfrak{S}_n})
\leq cci^{+}(\knx^{\mathfrak{S}_n})\leq \lfloor \frac{n}{2}\rfloor+1.$
\end{theorem}

\begin{proof} First we prove the claim that $cci^+(\knx^{\mathfrak{S}_n})
\leq \lfloor \frac{n}{2} \rfloor+1$.

Following the proof of Theorem \ref{zzthm3.14}, let $C_2$ be the
subalgebra of $\knx^{\mathfrak{S}_n}$ defined before Lemma \ref{zzlem2.4},
which is (isomorphic to) the iterated Ore extension
$$k[P_4,P_8,\cdots, P_{4 \lfloor \frac{n}{2}\rfloor}]
[P_1][P_3;\tau_{3},\delta_3]\cdots [P_{n'};\tau_{n'},\delta_{n'}]$$
where $n'=2\lfloor \frac{n-1}{2}\rfloor+1$. Let $F_{2n-3}$ be the
iterated Ore extension defined in the proof of Theorem
\ref{zzthm3.14}. (We are not going to use $F_{2n-1}$, instead we
will define two new algebras $H_{2n-1}$ and $H_{2n+1}$.) Recall from
the proof of Theorem \ref{zzthm4.15} that $p_i$ is the image of
$P_{2i}$ for all $i=1,\cdots,n$. By Lemma \ref{zzlem2.4}(5),
$P_{2i}$ are in $C_2$ for all $i$. Define
$H_{2n-1}=F_{2n-3}[Q_{2n-1}; \phi_{2n-1}]$ where $\phi_{2n-1}:
P_{i}\mapsto (-1)^{i(n-1)}P_i$ for all even $i$ and all odd $i\leq
2n-3$ if $P_{i}$ appeared in $F_{2n-3}$. It is easy to check that
$\phi_{2n-1}$ is an algebra automorphism of $F_{2n-3}$ and therefore
$H_{2n-1}$ is an iterated Ore extension. Define
$H_{2n+1}=H_{2n-1}[Q_{2n+1};\phi_{2n+1},\lambda_{2n+1}]$ where
$\phi_{2n+1}$ is an algebra automorphism determined by
$\phi_{2n+1}:\begin{cases} P_{i}\mapsto (-1)^{in}P_i &
{\text{for even $i$ or odd $i\leq 2n-3$}}\\
Q_{2n-1}\mapsto (-1)^{n+1} Q_{2n-1} & \end{cases}$ (see the proof
of Lemma \ref{zzlem4.13}(1)), and
$\phi_{2n+1}$-derivation $\lambda_{2n+1}$ is determined by
$$\lambda_{2n+1}:\begin{cases} P_{i}\mapsto 0 &
{\text{if $i$ is even and $i\leq 2n$}}\\
P_{i}\mapsto (-1)^{n+1}2n Q_{2n-1}f_{2i-2}(P_2,\cdots,P_{2n})&
{\text{if $i$ is odd and $i\leq 2n-3$}}\\
Q_{2n-1}\mapsto 0 & \end{cases},$$ where $f_{2i-2}$ is given by
\eqref{E3.7.1}. Similar to the proof of Lemma \ref{zzlem4.13}(2),
one can show that $\lambda_{2n+1}$ is a $\phi_{2n+1}$-derivation,
therefore $H_{2n+1}$ is an iterated Ore extension. Let
$u_s=P_{2s-1}^2-P_{4s-2}$ for all integers from $s=\lfloor
\frac{n-1}{2}\rfloor+2$ to $s=n-1$. Let $u_{n+1}$ be
$Q_{2n-1}^2-b_2$ where $b_2\in C_1\subset C_2$ is defined in
\eqref{E4.13.2}. Let $u_{n+2}$ be $Q_{2n+1}^2-b_1$ where $b_1\in
C_1\subset C_2$ is defined in \eqref{E4.13.1}.

The proof of Lemma \ref{zzlem3.11} (see also Lemma \ref{zzlem4.14})
shows that
$$\{u_{\lfloor \frac{n-1}{2}\rfloor+2}, \cdots, u_{n-1},
u_{n+1},u_{n+2}\}$$ is a central regular sequence of $H_{2n+1}$.
It is straightforward to see that
$$H_{2n+1}/(u_{\lfloor \frac{n-1}{2}\rfloor+2}, \cdots, u_{n-1},
u_{n+1},u_{n+2})\cong \knx^{\mathfrak{S}_n}.$$
Therefore $cci^+(\knx^{\mathfrak{S}_n})\leq
n+1-(\lfloor \frac{n-1}{2}\rfloor+1)= \lfloor \frac{n}{2}\rfloor+1$ and
we proved the claim.

By Theorem \ref{zzthm4.15}
$$\begin{aligned}
H_{\knx^{\mathfrak{A}_n}}(t)&= H_{\knx^{\mathfrak{S}_n}}(t)
\frac{(1+t^n)(1+t^{n-1})}{1+t^{2n-1}}\\
&=\frac{\prod_{s=\lfloor \frac{n-1}{2}\rfloor+2}^n (1-t^{4s-2})}
{\prod_{j=1}^{\lfloor \frac{n}{2}\rfloor}(1-t^{4j}) \prod_{i=1}^n
(1-t^{2i-1})}\;
\frac{(1-t^{2n-1})(1-t^{2n})(1-t^{2(n-1)})}{(1-t^{4n-2})(1-t^n)(1-t^{n-1})}\\
&=\frac{\prod_{s=\lfloor \frac{n-1}{2}\rfloor+2}^{n-1} (1-t^{4s-2})}
{\prod_{j=1}^{\lfloor \frac{n}{2}\rfloor}(1-t^{4j}) \prod_{i=1}^{n-1}
(1-t^{2i-1})}\;
\frac{(1-t^{2n})(1-t^{2(n-1)})}{(1-t^n)(1-t^{n-1})}\\
&=\frac{\prod_{s=\lfloor \frac{n-1}{2}\rfloor+2}^{n-1} (1-t^{4s-2})}
{\prod_{j=1}^{\lfloor \frac{n}{2}-1\rfloor}(1-t^{4j}) \prod_{i=1}^{n-1}
(1-t^{2i-1})}\;
\frac{(1-t^{2(2\lfloor\frac{n-1}{2}\rfloor+1)})}{(1-t^n)(1-t^{n-1})}
\end{aligned}
$$
which is an expression satisfying the condition in Definition
\ref{zzdef3.13}(2). Hence
$$\lfloor \frac{n}{2}\rfloor=cyc(\knx^{\mathfrak{S}_n})
\leq cci^{+}(\knx^{\mathfrak{S}_n})\leq \lfloor \frac{n}{2}\rfloor+1.$$
\end{proof}

\begin{question}\label{zzque4.17}
Let $A$ be either $\knx^{\mathfrak{S}_n}$ and
$\knx^{\mathfrak{A}_n}$. Let $E(A)$ be the $\Ext$-algebra
$\Ext^*_A(k,k)$.
\begin{enumerate}
\item
Is $E(A)$ noetherian?
\item
What is the GK-dimension of $E(A)$?
\end{enumerate}
\end{question}

\section{Converse of Kac-Watanabe-Gordeev Theorem}
\label{xxsec5}

Kac-Watanabe-Gordeev showed that when $\kpx^G$ is a complete
intersection then $G$ must be generated by classical bireflections.
We next prove the converse of this result for $\knx^G$ when
$G\subset \mathfrak{S}_n$ and note that the converse is not true for
$\kpx^G$. By Lemma \ref{xxlem1.7}(3) a quasi-bireflection must be a
2-cycle or a 3-cycle. We conclude by showing that for subgroups $G$
of $\mathfrak{S}_4$ acting on $k_{-1}[x_1, x_2, x_3, x_4]$, the
fixed subring $k_{-1}[x_1, x_2, x_3, x_4]^G$ is a cci if and only if
$G$ is generated by quasi-bireflections, and when $G$ is not
generated by quasi-bireflections, $k_{-1}[x_1, x_2, x_3, x_4]^G$ is
not cyclotomic Gorenstein, hence $k_{-1}[x_1, x_2, x_3, x_4]^G$ is
not any of the kinds of complete intersections described in
Definition \ref{xxdef1.8}. The following result on permutation
groups may be well-known, but is included for completeness.

\begin{proposition}
\label{xxpro5.1}
Let $G$ be a subgroup of $\mathfrak{S}_n$.
\begin{enumerate}
\item
If $G$ is generated by 3-cycles, then $G$ is an
internal direct product of alternating groups.
\item
If $G$ is generated by 3-cycles and 2-cycles, then
$G$ is an internal direct product of
alternating and symmetric groups.
\end{enumerate}
\end{proposition}

We first prove some lemmas. Let $X$ be any subset of
$\{i\}_{i=1}^n:=\{1,\cdots,n\}$. We use ${\mathfrak S}_{X}$ for the
full symmetric group of $X$.

\begin{proof}
Suppose that $G$ is generated by 3-cycles and 2-cycles.  We may
assume that $G=\langle \tau_1, \tau_2, \ldots , \tau_{\ell}\rangle $
where $\tau_1, \tau_2, \ldots , \tau_{\ell}$ are all of the 3-cycles
and 2-cycles in $G$. Let $X=\{1, 2, \ldots , n\}$.  We will show
that there are disjoint nonempty subsets $X_1, X_2, \ldots , X_k$ of
$X$ such that $G=G_1\times G_2 \times \cdots \times G_k$ where $G_i$
is the alternating or symmetric group on $X_i$. Given a permutation
$\sigma$ define $M(\sigma) = \{x \in X: \sigma(x) \neq x\}$, the set
of elements that are moved by $\sigma$.  Let ${\displaystyle Y =
\bigcup_{\sigma \in G} M(\sigma)}$ and define a relation $\sim$ on
$Y$ by $x\sim y$ if there exists 3-cycles and/or 2-cycles $\sigma_1,
\sigma_2, \ldots , \sigma_m$ such that $x\in M(\sigma_1), y \in
M(\sigma_m)$ and $M(\sigma_i)\cap M(\sigma_{i+1}) \neq \emptyset$
for $i = 1, 2, \ldots ,m-1.$ In this case we say that there is a
path from $x$ to $y$. It is easy to see that $\sim$ is an
equivalence relation on $Y$. Let $X_1, X_2, \ldots ,X_k$ be the
equivalence classes. We view the $X_i$ as the path connected
components of $Y$. Clearly either $M(\tau_j) \subseteq X_i$ or
$M(\tau_j)\cap X_i = \emptyset$ for all $i,j$. Let $G_i = \langle
\tau_j: M(\tau_j) \subseteq X_i \rangle$.

Case 1: Suppose that $G$ is generated by 3-cycles.  It will be
sufficient to show that each $G_i$ is an alternating group.
Furthermore, there is no loss of generality in assuming that there
is one component $Y$.  We will induct on $\ell$.  If $|Y|=3$, (the
smallest possible) then $G = \langle \tau \rangle \cong A_3.$ If
$|Y|=4$, we may assume that $Y=\{1,2,3,4\}, \tau_1=(1,2,3)$ and
$\tau_2=(2,3,4)$. In this case $|\langle \tau_1\rangle \langle
\tau_2\rangle| =9$ and $G$ must be all of $A_4$. Inductively assume
that whenever $G=\langle \tau_1, \tau_2, \ldots ,\tau_{\ell}\rangle$
has one component $Y$ with $4\leq |Y|=s\leq n$, then $G\cong A_s$.
We may let $Y=\{1, 2, \ldots , s\}.$  Now suppose that $G' = \langle
\tau_1, \tau_2, \ldots , \tau_{\ell+1}\rangle $ where
$\tau_{\ell+1}$ is a 3-cycle, and $\displaystyle Y' =
\bigcup_1^{\ell+1} M(\tau_{j})$ is connected. Let $\tau_{i_1},
\tau_{i_2}, \ldots , \tau_{i_m}$ be a maximal path in $Y'$.  Then
$\displaystyle \cup_{j\neq i_m} M(\tau_j)$ must be connected, for
otherwise, we could extend the path.  Hence there is no loss of
generality in assuming that $\tau_{\ell+1}$ is such that
$\displaystyle Y= \bigcup_{i \neq \ell+1} M(\tau_i)$ is connected
with $|Y| = s$. Let $G = \langle \tau_1, \tau_2, \ldots
,\tau_{\ell}\rangle$. There are two subcases.

Case 1.1: $|Y'| = s+1.$  We may assume, renumbering if necessary,
that $\tau_{\ell+1} = (s-1,s,s+1).$ We will show that $G'$ contains
all elements that are products of two disjoint 2-cycles.  The set of
all such generates a normal subgroup of $A_{s+1}$, and hence we
would have  $G' = A_{s+1}.$ By induction we have all disjoint
products $(i,j)(k,\ell)$ where $i, j, k, \ell \leq s.$ If $i,j \leq
s-1$, then $(i,j)(s-1,s)(s-1,s,s+1) = (i,j)(s,s+1)$. Then
$(s-1,k)(i,j)\cdot (i,j)(s,s+1) = (s-1,k)(s,s+1).$  The conjugation
$(k,s,\ell)(i,j)(s,s+1)(k,\ell,s) = (i,j)(\ell,s+1)$ gives the
remaining products.  Thus $G'=A_{s+1}$ and the result follows by
induction.

Case 1.2: $|Y'|=s+2.$ We may assume that $\tau_{\ell+1} = (s, s+1,
s+2).$ By induction $G$ is $A_s$ and we have the following chain
from $1$ to $s-1$:
\[ (1,2,3), (2,3,4), \ldots , (s-3, s-2, s-1). \]
Computing
\[ (1,2)(s-1,s)(s,s+1,s+2)(1,2)(s-1,s) = (s-1,s+1,s+2), \]
and $(s-1,s+1,s+2) \in G'.$  We have that $Y'' = \{1, 2, \ldots ,
s-1\}\cup \{s+1,s+2\}$ is a connected component, and by induction
$G'' = \langle A_{s-1}, (s-1,s+1,s+2) \rangle$ is a copy of
$A_{s+1}.$ Then $G=\langle G'', \tau_{\ell+1}\rangle$ is the
alternating group $A_{s+2}$ by Case 1.1.

Case 2: Once again there is no loss of generality in assuming that
there is one connected component.   We may also suppose that $G$
contains at least one 2-cycle by Case 1.  Again the proof is by
induction on $\ell$. If $|Y| =2,$ the result is clear. Since
$(1,2)(2,3) = (1,2,3)$, we see that if $|Y|=3$, then $G=S_3.$
Inductively assume that whenever $G = \langle \tau_1, \tau_2, \ldots
, \tau_{\ell}\rangle$ with $3\leq |Y| \leq n$ then $G$ is a
symmetric group.  Now suppose that $G' = \langle \tau_1, \tau_2,
\ldots , \tau_{\ell+1}\rangle$ with $\displaystyle Y' =
\bigcup_1^{\ell+1} M(\tau_j)$ connected.  Again we may assume that
$\displaystyle Y=\bigcup_{j\neq \ell+1}M(\tau_j)$ is connected with
$|Y|=s \leq n$.  Then by induction or by case 1 we have that $G
=\langle \tau_1, \tau_2, \ldots ,\tau_{\ell}\rangle$ is either a
symmetric group or an alternating group (if all $\tau_i$ for $i \leq
\ell$ are 3-cycles). We have two subcases.

Case 2.1: $\tau_{\ell+1}$ is a 3-cycle.  By the argument in Case 1,
$G'$ contains the full alternating group.  Since $G'$ must also
contain a 2-cycle, it is the full symmetric group.

Case 2.2: $\tau_{\ell+1}$ is a 2-cycle.  Without loss of generality
we may assume that $\tau_{\ell+1} = (s, s+1).$  As noted, $G$ is
either the symmetric group or the alternating group.  In this case
$G'$ must contain
\[ (1,2)(s-1,s)(s,s+1) = (1,2)(s-1,s, s+1). \]
Squaring yields that $(s-1, s+1, s) \in G'$.  By Case 1.1, $G'$
contains the full alternating group. Since it also contains a
2-cycle, it must be the full symmetric group.

The result follows by induction.
\end{proof}

Let $A$ and $B$ be two graded algebra. Define $A\otimes_{-1} B$ be
the ${\mathbb Z}^2$-graded twist of the tensor product $A\otimes B$
by the twisting system
$$\sigma:=\{\sigma_{i,j}=Id^{i} \xi_{-1}^j \mid (i,j)\in {\mathbb
Z}^2\}$$ where $\xi_{-1}$ maps $a\otimes b\mapsto (-1)^{|a|+|b|}
a\otimes b$ for all $a\otimes b\in A\otimes B$. The following lemmas
are easy to check.

\begin{lemma}
\label{xxlem5.2} Retain the above notation.
\begin{enumerate}
\item
$A\otimes_{-1} B=A\otimes B$ as ${\mathbb Z}^2$-graded vector spaces.
\item
Identifying $A$ with $A\otimes 1\subset A\otimes_{-1} B$ and
identifying $B$ with $1\otimes B\subset A\otimes_{-1} B$. Then
$A$ and $B$ are subalgebras of $A\otimes_{-1} B$,  and the algebra
$A\otimes_{-1} B$ is equal to the vector space generated by the products
$AB$ {\rm{(}}and $BA$ respectively{\rm{)}}.
\item
Under the identification in part (2),
$a b=(-1)^{|a|\; |b|} b a$ for all $a\in A$ and $b\in B$.
\end{enumerate}
\end{lemma}

\begin{lemma}
\label{xxlem5.3} Let $m<n$.
\begin{enumerate}
\item
$k_{-1}[x_1,\cdots,x_m]\otimes_{-1} k_{-1}[x_{m+1},\cdots, x_n]
\cong k_{-1}[x_1,\cdots,x_{n}]$.
\item
If $G_1\subset \Aut(A)$ and $G_2\subset \Aut(B)$, then
$(A\otimes_{-1} B)^{G_1\times G_2}=A^{G_1}\otimes_{-1} B^{G_2}$.
\item\cite[Lemma 2.7]{KKZ3}
If $A$ and $B$ are AS regular, then so is $A\otimes_{-1} B$.
\item
Suppose $A=R/(\Omega_{1},\cdots, \Omega_{m})$ and
$B=C/(f_1,\cdots,f_d)$ where $R$ and $C$ are AS regular and
$\{\Omega_i\}_{i=1}^m$ and $\{f_j\}_{j=1}^d$ are regular normal
sequences of positive even degrees. If $R\otimes C$ is noetherian,
then $A\otimes_{-1} B$ is a factor ring of a noetherian AS regular
algebra modulo a regular normal sequences of positive even degrees.
As a consequence, $A\otimes_{-1} B$ is a cci.
\end{enumerate}
\end{lemma}

For any subset $X$ of $[1,\cdots,n]$, let $\mathfrak{S}_X$ denote
the symmetric group of $X$ (all permutations of $X$).

\begin{theorem}
\label{xxthm5.4} If $G$ is a subgroup of $\mathfrak{S}_n$ generated
by quasi-bireflections, then $\knx^G$ is a cci.
\end{theorem}

\begin{proof} We use induction on $n$. Suppose the assertion holds
for $G\subset \mathfrak{S}_m$ for all $m\leq n-1$. Now let $G$ be
a subgroup of $\mathfrak{S}_n$ generated by quasi-bireflections.
If $G$ is $\{1\}$, the assertion is trivial. If $G=\mathfrak{S}_n$
or $\mathfrak{A}_n$, the assertion follows from Theorems
\ref{zzthm3.12} and \ref{zzthm4.15}. Otherwise, by Proposition
\ref{xxpro5.1}, there is a disjoint union $X\cup Y=[1,\cdots,n]$
such that $G$ is a product of $G_1$ and $G_2$, where
$G_1$ and $G_2$ are subgroups $\mathfrak{S}_X$ and $\mathfrak{S}_Y$
respectively, and further $G_1$ is either $\mathfrak{S}_X$ or
$\mathfrak{A}_X$ and $G_2$ is generated by quasi-bireflections of
$k_{-1}[x_i\mid i\in Y]$ (or equivalently, $2$- or 3-cycles of
$\mathfrak{S}_Y$). By induction, both $A^{G_1}$ and $B^{G_2}$ are
cci, where $A=k_{-1}[x_i\mid i\in X]$ and $B=k_{-1}[x_i\mid i\in Y]$.
It follows from Lemma \ref{xxlem5.2} and \ref{xxlem5.3} that
$\knx^G\cong A^{G_1}\otimes_{-1}B^{G_2}$ is a cci.
\end{proof}

The following example shows that for $\kpx$
permutation groups generated by classical bireflections need not
have a fixed ring that is a complete intersection.

\begin{example}
\label{xxex5.5} Let $\mathfrak{S}_5$ act on $A := k[x_1, x_2, x_3,
x_4, x_5]$ by permuting the variables. Let $G =\langle (1,2)(3,4),
(2,3)(4,5) \rangle$. These two generators are classical
bireflections. Note that $(1,2)(3,4)\cdot (2,3)(4,5) = (1, 2, 4, 5,
3)$.  Calculating shows that $\langle (1,2)(3,4),
(1,2,4,5,3)\rangle$ is a copy of the dihedral group $D_5$ of order
$10$ and is in fact all of $G$.  Using Molien's Theorem we have
\begin{eqnarray*}
H_{A^G}(t)&=& \frac{1}{10}\left(\frac{1}{(1-t)^5}
+\frac{5}{(1-t)^3(1+t)^2} +\frac{4}{1-t^5}\right) \\
&=& \frac{t^6-t^5+2t^3-t+1}{(1-t)^2(1-t^2)^2(1-t^5)}.
\end{eqnarray*}
The numerator is an irreducible polynomial that is not cyclotomic;
in fact, none of its zeros are roots of unity.  Hence $A^G$
cannot be a complete intersection.
\end{example}

We conclude by computing the invariants of $A=k_{-1}[x_1, x_2. x_3,
x_4]$ under each of the subgroups of $\mathfrak{S}_4$. In this case
the conjectured generalization of the Kac-Watanabe-Gordeev Theorem
becomes both necessary and sufficient.  We show that $A^H$ is a
cci if and only if $H$ is generated by
quasi-bireflections (i.e. 2-cycles or 3-cycles); when $H$ is not
generated by quasi-bireflections $A^H$ is not cyclotomic
Gorenstein -- hence not any kind of complete intersection by Theorem
\ref{xxthm1.10}.

\begin{example}
\label{xxex5.6} For the following subgroups $H$ of $\mathfrak{S}_4$
we consider the fixed subring $A^H$. We show that $A^H$ is
either a cci or not cyclotomic
Gorenstein (and hence none of the kinds of complete intersection we
considered in Definition \ref{xxdef1.8}).
\begin{itemize}
\item
If $H$ is the full symmetric group or the alternating group, both
generated by quasi-bireflections, we have shown that $A^H$ is a
cci. Similarly, cyclic subgroups
generated by a 2-cycle (so isomorphic to $\mathfrak{S}_2$) or by a
3-cycle (so isomorphic to $\mathfrak{A}_3$) are also easily seen to
give ccis when they act on
$A=k_{-1}[x_1, x_2, x_3, x_4]$ (we showed they did when they acted
on  $A=k_{-1}[x_1, x_2]$ and $A=k_{-1}[x_1, x_2, x_3]$ and the
results extend by fixing the remaining variable(s)).
\item
Let $H$ be the subgroup of order 2 generated by an element that is a
product of two disjoint 2-cycles, e.g. $(12)(34)$; this subgroup is
not generated by quasi-bireflections of $A$ (it is generated by a
bireflection of $k[x_1, x_2, x_3, x_4]$). Molien's Theorem shows
that the Hilbert series of $A^H$ is
$$\frac{1-2t+4t^2-2t^3+t^4}{(1-t)^4 (1+t^2)^2}$$
which has zeros that are not roots of unity.  Hence $A^H$ is not
cyclotomic Gorenstein.
\item
We have already noted (Example \ref{xxex1.6}) that the subgroup $H$
generated by a 4-cycle is not generated by quasi-bireflections, and
that the invariants $A^H$ are not cyclotomic Gorenstein.
\item
Let $H$ be the Klein-Four subgroup generated by two disjoint
2-cycles (e.g. $H = \langle (12), (34) \rangle$).  Then $H$ is
generated by quasi-bireflections of $A$, the generators of $A^H$
are $x_1+x_2, x_3+x_4, x_1^3 +x_2^3, x_3^3+x_4^3$, Hilbert series of
$A^H$ is
$$\frac{1-t+t^2}{(1-t)(1+t^2)^2},$$
and
$$A^H \cong \frac{k[p_1, p_2, q_1,q_2][y_1][y_2; \tau_1, \delta_1]
[z_1; \tau_2, \delta_2][z_2;\tau_3, \delta_3] }{\langle y_1^2-a_1,
y_2^2-a_2,z_1^2 - b_1, z_2^2-b_2 \rangle},$$
where $p_1$, $p_2$
(resp., $q_1, q_2$) correspond to the first two symmetric
polynomials in $x_1^2, x_2^2$ (resp., $x_3^2, x_4^2$), $y_1, y_2$
(resp., $z_1, z_2$) correspond to $x_1+x_2, x_1^3+x_2^3$ (resp.,
$x_3+x_4, x_3^3+x_4^3$).
\item
The Klein-Four subgroup of even permutations
$$H= \{1, (12)(34), (13)(24), (14)(23)\},$$
which is not generated by quasi-bireflections
of $A$.  The Hilbert series of $A^H$ is
$$\frac{1-3t+5t^2-3t^3+t^4}{(1-t)^4(1+t^2)^2},$$
so $A^H$ is not cyclotomic Gorenstein.
\item
Let $H$ be a subgroup $\mathfrak{S}_4$ of order 6.  Then $H$ is
isomorphic to the symmetric group $\mathfrak{S}_3$, without loss of
generality of the form $H = \langle (123),(12) \rangle$.  This group
is generated by quasi-bireflections, and $A^H$ is a complete
intersection (we showed this for $k_{-1}[x_1, x_2, x_3]$ and the
extension to $A$ is not difficult).
\item
Let $H$ be a dihedral group of order 8 (a Sylow-2 subgroup of
$\mathfrak{S}_4$).  Then $H$ is of the form
$$D_4 = \{1, (1234), (13)(24), (1432), (13),
(24), (12)(34), (14)(23) \},$$
so not generated by quasi-bireflections.  The Hilbert series
of the fixed subring is
$$\frac{1-3t+5t^2-5t^3+5t^4-5t^5+5t^6-3t^7+t^8
}{(1-t)^4(1+t^4)(1+t^2)^2}$$
$$=\frac{(1-t+t^2)(1-2t+2t^2-t^3+2t^4-2t^5+t^6)
}{(1-t)^4(1+t^4)(1+t^2)^2}$$
so $A^H$ is not cyclotomic Gorenstein.
\end{itemize}

Note:  It might be nice to know degrees of generators and how they
compare to $n^2 = 16$.
\end{example}

\begin{question}\label{xxque5.7}
For $H$ a subgroup of $\mathfrak{S}_n$, is $\knx^H$ a cci if and
only if $H$ is generated by quasi-bireflections?
\end{question}

\section{Appendix}
\label{xxsec6}

In this section we find generating functions for the class of
restricted partitions having no repeated odd parts and the class
having no repeated even parts.  It is included since we were unable
to find them in the literature.

Let $d_n(k)$ be the number of partitions of $k$ with at most $n$
parts having no repeated odd parts.  Make the convention that
$d_n(1)= 1$ and $d_n(\ell) =0$ for $\ell < 0.$  Let $D_n(t)$ be the
corresponding generating function
\[ D_n(t) = \sum_{k=0}^{\infty} d_n(k) t^k. \]
 There is only one way to partition $k$ into $1$ part, so
\begin{eqnarray*} D_1(t)& = & 1 + t +t^2+t^3 + \cdots + t^k + \cdots \\
& = & \frac{1}{1-t} \\
& = & \frac{1-t^2}{(1-t)(1-t^2)}
\end{eqnarray*}

We will now try to find a recurrence relation for $d_n(k).$  We will
write a partition $\cal P$ of $k$ having at most $n$ parts as ${\cal
P} = p_1, p_2, \ldots ,p_n$ where $p_1 \geq p_2 \geq \ldots \geq
p_n$ and $k = p_1 + p_2 + \cdots + p_n.$  Let ${\cal D}_{n,k} =
\{{\cal P} = p_1, p_2, \ldots ,p_n : {\rm
 with \, no \, repeated \, odd \, parts}\}.$  Then we have
 \[ {\cal D}_{n,k} = \{{\cal P}: p_n =0\} \cup_d
 \{{\cal P}: p_n=1\} \cup_d \{{\cal P}: p_n \geq 2\}. \]
 \begin{itemize}
 \item
 Clearly $|\{{\cal P}: p_n =0\}| = d_{n-1}(k).$
 \item
 If $p_n=1$, consider the association
 \[ {\cal P} \mapsto {\cal P}' = p_1-2,p_2-2, \ldots ,p_{n-1}-2, 0. \]
Since $p_{n-1} > p_n =1,$ this will be a partition of $k-1-2(n-1) =
k-2n+1.$  Since parity is preserved there will be no repeated odd
parts, and every such partition of $k-2n+1$ can occur in this
manner.  Hence
 $|\{{\cal P}:p_n=1\}| = d_{n-1}(k-2n+1).$
 \item
 If $p_n \geq 2,$ consider the association
 \[ {\cal P} \mapsto {\cal P}' = p_1-2, p_2-2, \ldots , p_n-2. \]
This will be a partition of $k-2n$ with no repeated odd parts.  Once
again every such partition can occur in this manner.  Hence
$|\{{\cal P}:p_n \geq 2\}| = d_n(k-2n).$
 \end{itemize}
 This yields the following recurrence relation
 \begin{eqnarray*}
 d_n(k) = d_{n-1}(k) + d_{n-1}(k-2n+1) + d_n(k-2n).
 \end{eqnarray*}
In terms of generating functions we have
  \begin{eqnarray*}
 D_n(t) &= &D_{n-1}(t) + D_{n-1}(t)t^{2n-1}+D_n(t)t^{2n}.
 \end{eqnarray*}
This gives the recurrence
\begin{eqnarray*}
 D_n(t)& =& D_{n-1}(t) \frac{(1+t^{2n-1})}{(1-t^{2n})} \\
 &=& D_{n-1}(t) \frac{(1-t^{4n-2})}{(1-t^{2n-1})(1-t^{2n})}.
 \end{eqnarray*}

Using this last recurrence relation  a simple induction argument
proves the following Proposition.

 \begin{proposition} \label{xxpro6.1} The generating function $D_n(t)$
 for the number of partitions
 with at most $n$ parts having no repeated odd parts is given by
 \begin{eqnarray*}
 D_n(t)&=& \frac{(1-t^2)(1-t^6)(1-t^{10})\cdots (1-t^{4n-2})}{
 (1-t)(1-t^2)(1-t^3) \cdots (1-t^{2n-1})(1-t^{2n})}.
 \end{eqnarray*}
 \end{proposition}

\begin{remark}
\label{xxrem6.2} We note using the
$${\text{\it Online Encyclopedia of Integer Sequences}
\;\; (\text{http://oeis.org/)}}$$ for specific values of $n$ we
found that $D_n(t)$, the Hilbert series of $\knx^{\mathfrak{S}_n}$,
is also the Hilbert series of the invariants of $\mathcal{A}  =
k[y_1 \ldots, y_n] \otimes E(e_1, \ldots, e_n)$ under the action of
$\mathfrak{S}_n$, where $k$ is any field of characteristic not equal
to two, the degree of each $y_i = 2$, $E(e_1, \ldots, e_n)$ is the
exterior algebra on elements $e_i$ of degree 1, and $\mathfrak{S}_n$
acts on both $k[y_1 \ldots, y_n]$ and $E(e_1, \ldots, e_n)$ by
permutations. (See \cite[pp. 110-11]{AM}). We note that one can
filter $\knx$ by letting $I$ be the ideal generated by $\{x_1^2,
\ldots, x_n^2\}$.  Then the associated graded algebra
$${\rm gr}(\knx) = \knx/I \oplus I/I^2 \oplus I^2/I^3 \oplus
\cdots \oplus I^m/I^{m+1} \oplus \cdots$$
is isomorphic as a graded algebra to $\mathcal{A}$ under the map
that associates $y_i \mapsto x_i^2 +I^2$ and $e_i \mapsto x_i +I$.
Further the action of $\mathfrak{S}_n$ on $\knx$ extends to an
action on ${\rm gr}(\knx)$, and
$$\mathcal{A}^{\mathfrak{S}_n} \cong
{\rm gr}(\knx)^{\mathfrak{S}_n} \cong {\rm
gr}(\knx^{\mathfrak{S}_n}).$$ Since ${\rm
gr}(\knx^{\mathfrak{S}_n})$ has the same Hilbert series as
$\knx^{\mathfrak{S}_n}$, it follows that $D_n(t)$ is the Hilbert
series of $\knx^{\mathfrak{S}_n}$.
\end{remark}

Let $s_n(k)$ be the number of partitions of $k$ with at most $n$
parts having no repeated even parts (not even repeated 0 parts), and
let $S_n(t)$ be the corresponding generating function. The purpose
of this section is to find $S_n(t)$.

First we briefly consider a slight variation.  Let $w_n(k)$ be the
number of partitions of $k$ with exactly $n$ nonzero parts having no
repeated even parts, and let $W_n(t)$ be the corresponding
generating function.  Let $\cal P$ be such a partition. Correspond
to $\cal P$ the partition ${\cal P} \mapsto {\cal P}' = p_1-1,
p_2-1, \ldots , p_n-1$.  This will be a partition of $k-n$ with at
most $n$ parts having no repeated odd parts, and any such partition
can occur in this manner.  Hence $w_n(k) = d_n(k-n)$, and $W_n(t) =
t^nD_n(t).$

Let ${\cal S}_{n,k}$ be the collection of all partitions of $k$ with
at most $n$ parts having no repeated even parts.  Then we have
\[ {\cal S}_{n,k} = \{{\cal P}: p_n =0\} \cup_d
\{{\cal P}: p_n=1\} \cup_d \{{\cal P}: p_n \geq 2\}. \]
Since there
are no repeated empty parts, the partitions in the first set will be
partitions having exactly $n-1$ nonzero parts and $|\{{\cal P}:
p_n=0\}| = w_{n-1}(k).$  For each partition $\cal P$ in the second
set we correspond ${\cal P} \mapsto {\cal P}' = p_1, p_2, \ldots ,
p_{n-1},0$, which will be a partition of $k-1$ with exactly $n-1$
nonzero parts and no repeating even parts.  Since all such occur in
this manner, we have $|\{{\cal P}: p_n=1\}| = w_{n-1}(k-1).$
Similar to the no repeated odd case we see that $|\{{\cal P}: p_n
\geq 2\}| = s_n(k-2n).$  This gives the recurrence relation
\[
s_n(k) = w_{n-1}(k) + w_{n-1}(k-1) + s_n(k-2n)
\]
In terms of generating functions we have
\[S_n(t) = W_{n-1}(t) + W_{n-1}(t) t + S_n(t) t^{2n}, \]
and
\begin{equation} \label{E6.2.1}\tag{E6.2.1}
S_n(t) = W_{n-1}(t) \frac{(1+t)}{(1-t^{2n})} = D_{n-1}(t)
\frac{(1+t)t^{n-1})}{(1-t^{2n})}.
\end{equation}

Summarizing we have the following Proposition.

\begin{proposition} \label{xxpro6.3}
If $S_n(t)$ is the generating function for the number of partitions
having at most $n$ parts with no repeated even parts, then
\begin{eqnarray*}
S_n(t)&=& \frac{(1-t^2)(1-t^6)(1-t^{10})\cdots
(1-t^{4n-2})t^{n-1}(1+t)}{(1-t)(1-t^2)(1-t^3) \cdots
(1-t^{2n-1})(1-t^{2n})(1+t^{2n-1})}.
 \end{eqnarray*}
and
\[ S_n(t) = D_n(t)\frac{t^{n-1}(1+t)}{(1+t^{2n-1})}. \]
\end{proposition}

\begin{proof}
From \eqref{E6.2.1} we have
$$\begin{aligned}
S_n(t)&= D_{n-1}(t)\frac{t^{n-1}(1+t)}{(1-t^{2n})} \\
&= \left(\frac{(1-t^2)(1-t^6)\cdots (1-t^{4n-6})}{(1-t)(1-t^2)
\cdots (1-t^{2n-2})}\right)
\left(\frac{t^{n-1}(1+t)}{(1-t^{2n})}\right) \\
& = \left(\frac{(1-t^2)\cdots (1-t^{4n-6})}{(1-t) \cdots
(1-t^{2n-2})}\right) \left(\frac{t^{n-1}(1+t)}{(1-t^{2n})}\right)
\frac{(1-t^{2n-1})(1+t^{2n-1})}{(1-t^{2n-1})(1+t^{2n-1})} \\
& =  \frac{(1-t^2)(1-t^6)(1-t^{10})\cdots
(1-t^{4n-2})t^{n-1}(1+t)}{(1-t)(1-t^2)(1-t^3)
\cdots (1-t^{2n-1})(1-t^{2n})(1+t^{2n-1})} \\
& = D_n(t)\frac{t^{n-1}(1+t)}{(1+t^{2n-1})}.
\end{aligned}
$$
\end{proof}

\subsection*{Acknowledgments}
Some of the results (e.g., Theorems \ref{zzthm3.10} and
\ref{zzthm4.10}) in this paper appeared in the Master Thesis
\cite{CA} of James Cameron Atkins, under the direction of E.
Kirkman. The authors thank James Cameron Atkins for the analysis
given in his Thesis on which Sections 3 and 4 of this paper are
based. E. Kirkman was partially supported by the Simons Foundation
(grant no. 208314)
and J.J. Zhang was partially supported by the National Science
Foundation (NSF DMS 0855743).

\end{document}